\documentclass[10pt, reqno]{amsart}

\usepackage{amsmath,amssymb,amsthm,amsfonts,verbatim}
\usepackage{microtype}
\usepackage[all,2cell]{xy}
\usepackage{mathtools}
\usepackage{graphicx}
\usepackage{pinlabel}
\usepackage{hyperref}
\usepackage{mathrsfs}
\usepackage{color}
\usepackage[dvipsnames]{xcolor}
\usepackage{enumerate}
\usepackage{cite}
\usepackage{soul}

\CompileMatrices

\usepackage[top=1.2in,bottom=1.8in,left=1.2in,right=1.2in]{geometry}

\usepackage{hyperref} 
\hypersetup{          
	colorlinks=true, breaklinks, linkcolor=[RGB]{51 102 204}, filecolor=Orchid, urlcolor=[RGB]{51 102 204},
	citecolor=Orchid, linktoc=all, }
\usepackage[nameinlink]{cleveref}

\theoremstyle{plain}
\newtheorem{theorem}{Theorem}[section]
\newtheorem{maintheorem}{Theorem}

\newtheorem{question}[theorem]{Question}

\newtheorem{proposition}[theorem]{Proposition}
\newtheorem{lemma}[theorem]{Lemma}

\newtheorem{corollary}[theorem]{Corollary}

\theoremstyle{definition}
\newtheorem{definition}[theorem]{Definition}

\newtheorem{remark}[theorem]{Remark}

\newtheorem{convention}[theorem]{Convention}

\newcommand{\nc}{\newcommand}
\nc{\dmo}{\DeclareMathOperator}

\nc{\Q}{\mathbb{Q}}
\nc{\F}{\mathbb{F}}
\nc{\R}{\mathbb{R}}
\nc{\Z}{\mathbb{Z}}
\nc{\C}{\mathbb{C}}
\nc{\N}{\mathbb{N}}
\nc{\Ell}{\mathcal{L}}
\nc{\M}{\mathcal{M}}
\nc{\K}{\mathcal{K}}
\nc{\I}{\mathcal{I}}
\nc{\T}{\mathcal T}
\nc{\U}{\mathcal U}
\nc{\disk}{\mathbb{D}}
\nc{\hyp}{\mathbb{H}}

\nc{\CP}{\mathbb{CP}}
\nc{\RP}{\mathbb{RP}}
\nc{\cS}{\mathcal{S}}
\dmo{\Mod}{Mod}
\dmo{\PMod}{PMod}
\dmo{\LMod}{LMod}
\dmo{\Diff}{Diff}
\dmo{\Homeo}{Homeo}
\dmo{\dist}{dist}
\dmo\BDiff{BDiff}
\dmo\SO{SO}
\dmo\Hom{Hom}
\dmo\SL{SL}
\dmo\Sp{Sp}
\dmo\rank{rank}
\dmo\sig{sig}
\dmo\Out{Out}
\dmo\Aut{Aut}
\dmo\Inn{Inn}
\dmo\GL{GL}
\dmo\PSL{PSL}
\dmo\EU{EU}
\dmo\BHomeo{BHomeo}
\dmo\EHomeo{EHomeo}
\dmo\EDiff{EDiff}
\dmo\Disc{Disc}
\nc\Sig{\Sigma}
\dmo\Teich{Teich}
\dmo\Fix{Fix}
\nc{\pair}[1]{\left\langle #1 \right\rangle}
\nc{\abs}[1]{\left| #1 \right|}
\nc{\action}{\circlearrowright}
\nc{\norm}[1]{\left | \left | #1 \right | \right |}
\nc{\abcd}[4]{\left(\begin{array}{cc} #1 & #2 \\ #3 & #4 \end{array}\right)}
\nc{\into}\hookrightarrow
\dmo{\Isom}{Isom}
\nc{\normal}{\vartriangleleft}
\dmo{\Vol}{Vol}
\dmo{\im}{Im}
\dmo{\Push}{Push}
\dmo{\Conf}{Conf}
\dmo{\PConf}{PConf}
\dmo{\id}{id}
\dmo{\Jac}{Jac}
\dmo{\Pic}{Pic}
\dmo{\Stab}{Stab}
\dmo{\Arf}{Arf}
\dmo{\End}{End}
\dmo{\Gal}{Gal}
\dmo{\lcm}{lcm}
\dmo{\ab}{ab}
\dmo{\opp}{op}
\dmo{\SU}{SU}
\dmo{\OT}{\Omega \mathcal{T}}
\dmo{\OM}{\Omega \mathcal{M}}
\dmo{\spin}{spin}
\dmo{\even}{even}
\dmo{\odd}{odd}
\dmo{\comp}{\mathcal{H}}
\dmo{\Mgk}{\mathcal{M}_{g, \underline{\kappa}}}
\dmo{\orb}{orb}
\dmo{\AJ}{AJ}
\dmo{\Ck}{\mathsf{C}(\underline{\kappa})}
\dmo{\Int}{Int}
\dmo{\pr}{pr}
\dmo{\lab}{lab}
\dmo{\Sym}{Sym}
\dmo{\Ann}{Ann}
\dmo{\Rad}{Rad}

\nc{\Span}[1]{\operatorname{Span}(#1)}

\renewcommand{\epsilon}{\varepsilon}
\renewcommand{\tilde}{\widetilde}
\renewcommand{\le}{\leqslant}
\nc{\coloneq}{\mathrel{\mathop:}\mkern-1.2mu=}
\nc{\margin}[1]{\marginpar{\scriptsize #1}}
\nc{\para}[1]{\medskip\noindent\textbf{#1.}}
\definecolor{myblue}{RGB}{102,153, 255}
\definecolor{myred}{RGB}{204,0,0}
\definecolor{mygreen}{RGB}{0,204,0}
\definecolor{myorange}{RGB}{255,102,0}
\nc{\red}[1]{\textcolor{myred}{#1}}
\nc{\blue}[1]{\textcolor{myblue}{#1}}
\nc{\proofof}[1]{\noindent {\em Proof (of #1).}}

\nc{\lb}{[}
\nc{\rb}{]}

\title{Simple closed curves in stable covers of surfaces}

\author{Nick Salter}
\email{nsalter@nd.edu}
\thanks{The author is supported by NSF Award No. DMS-2003984.}
\address{Department of Mathematics, University of Notre Dame, Hurley Hall, Notre Dame, IN 46556}
\date{September 27, 2021}

\begin{document}
\maketitle	
\begin{abstract}
    Let $f: X \to Y$ be a regular covering of a surface $Y$ of finite type with nonempty boundary, with finitely-generated (possibly infinite) deck group $G$. We give necessary and sufficient conditions for an integral homology class on $X$ to admit a representative as a connected component of the preimage of a nonseparating simple closed curve on $Y$, possibly after passing to a ``stabilization'', i.e. a $G$-equivariant embedding of covering spaces $X \hookrightarrow X^+$.
\end{abstract}
	
\section{Introduction}

Let $X$ be an oriented surface of finite type. It is a classical fact that $H_1(X;\Z)$ is spanned by {\em geometric} classes, i.e. elements $v \in H_1(X;\Z)$ represented as $v = [\gamma]$ for $\gamma \subset X$ some connected simple closed curve. Moreover, there is a simple algebraic criterion to determine if $v \in H_1(X;\Z)$ admits such a representative: if $X$ has at most one boundary component, it is necessary and sufficient that $v$ be {\em integrally primitive}: any expression of the form $v = k v'$ necessarily has $\abs{k} = 1$. 

The situation becomes vastly more complicated when one moves to the {\em relative} setting, and considers a covering $f: X \to Y$ of surfaces, typically regular with covering group $G$. Here, one is interested in the class of {\em relatively geometric} elements: classes $v \in H_1(X;\Z)$ represented by a connected simple closed curve $\gamma$ for which moreover $f(\gamma)$ is simple on $Y$. The facts mentioned above now become questions:
\begin{question}\label{question:span}
For which covers $f: X \to Y$ is the span $H_1(X;A)^{scc}$ of relatively-geometric classes equal to all of $H_1(X;A)$, where $A$ is an abelian group (typically $A = \Z, \Q, \C$)?
\end{question}
\begin{question}\label{question:individual}
For a cover $f: X \to Y$, what is an algebraic characterization of the relatively geometric elements of $H_1(X;\Z)$?
\end{question}
Aside from the intrinsic merit of these questions as being foundational to the study of equivariant geometric topology, they have been encountered in the study of the moduli space of curves and are closely intertwined with the {\em Ivanov conjecture} and with the representation theory of the mapping class group and its connection to the theory of arithmetic groups. See, e.g. \cite{PW,MP,FH2,looijenga,GLLM} for further discussion of these topics.

In recent years, substantial (although by no means exhaustive) progress has been made on \Cref{question:span} for the class of finite covering groups: work of Farb--Hensel, Koberda--Santharoubane, and Malestein--Putman \cite{FH,KS,MP} has highlighted the important role played by the representation theory of $G$, and has furnished examples of covers for which $H_1(X;A)^{scc}$ is a strict subgroup of $H_1(X;A)$, both for $A = \Z$ and $A=\Q$. See also \cite{flamm} for an investigation of the analogue of \Cref{question:individual} for graphs.

In light of the delicate nature of the emerging answer to \Cref{question:span}, it might seem overly audacious to attack the much more refined question posed in \Cref{question:individual}. However, the purpose of this paper is to do precisely this, with the caveat that our answer requires us to {\em stabilize} our covers by equivariantly embedding them in certain larger covers (see \Cref{definition:stabilization} for our precise conventions). In the stable setting, we are able to completely characterize relatively geometric classes that project onto {\em nonseparating} curves in $Y$ (it is an initially-surprising fact that a nonseparating curve $\gamma \subset X$ can project onto a separating curve in $Y$; the analysis of this class of curves is substantially more intricate and is postponed to future work). Remarkably, our techniques impose no requirements on the covering group $G$ whatsoever, other than that $G$ admit a surjection from a fundamental group of a surface - that is to say, $G$ can be an arbitrary finitely-generated group.

We identify four purely algebraic conditions necessary for a class $v \in H_1(X;\Z)$ to be relatively geometric and nonseparating in the above sense; these are given in \Cref{maintheorem:necessary} (see the paragraph following \Cref{maintheorem:necessary} for an overview of their meaning). Our main result \Cref{maintheorem:sufficient} shows that these conditions are sufficient in the stable setting.

\begin{maintheorem}[Relative geometricity: stably-sufficient conditions]\label{maintheorem:sufficient}
Let $Y$ be a connected oriented surface of finite type and nonempty boundary, let $G$ be a finitely-generated group and let $f: X \to Y$ be a connected regular $G$-covering classified by a surjective homomorphism $\phi: \pi_1(Y) \to G$. Let $v \in H_1(X;\Z)$ be given with with $f_*(v) \in H_1(Y;\Z)$ nonseparating. Then there exists a stabilization $f^+: X^+ \to Y^+$ on which $v$ is relatively geometric if and only if the necessary conditions of \Cref{maintheorem:necessary} hold. 
\end{maintheorem}

\begin{maintheorem}[Relative geometricity: stably-necessary conditions]\label{maintheorem:necessary}
Let $Y$ be a connected oriented surface of finite type and nonempty boundary, let $G$ be a finitely-generated group and let $f: X \to Y$ be a connected regular $G$-covering classified by a surjective homomorphism $\phi: \pi_1(Y) \to G$. Suppose $v \in H_1(X;\Z)$ is relatively geometric and that $f_*(v) \in H_1(Y;\Z)$ is nonseparating. Then there is a stabilization $f^+: X^+ \to Y^+$ such that the following must hold:
        \begin{enumerate}
            \item\label{condition:isotropy} $\pair{v,v} = 0$ (isotropy)
            \item\label{condition:parity} $q(v) = 0$ (parity)
            \item\label{condition:ideal} $I_v = \Z[G] \zeta_v$ (primitivity)
            \item\label{item:stabgen} $G_v = \pair{\phi_{v,*}(\pi_{v,*}(v)/\abs{G_v})}$ (coherence)
        \end{enumerate}
\end{maintheorem}

\para{Necessary conditions: overview} Here we give brief descriptions of the four conditions of \Cref{maintheorem:necessary}; see \Cref{section:pairing,section:q,section:necessary} for full details. The isotropy and primitivity conditions both are formulated in terms of a {\em relative intersection pairing} $\pair{\cdot,\cdot}$: this is a $\Z[G]$-valued skew-Hermitian form defined and studied in \Cref{section:pairing}. The relative intersection pairing records intersections not just between fixed homology classes, but their $G$-orbits. The isotropy condition then is simply a reflection of the fact that components of the $G$-orbit of a relatively geometric class are disjoint. The primitivity condition concerns the {\em pairing ideal} $I_v$ (defined as the left ideal in $\Z[G]$ of elements of the form $\pair{u,v}$ for $u \in H_1(X;\Z)$ arbitrary), and asserts that this must be a certain left-principal ideal; in the case where the stabilizer $G_v$ of $v$ is trivial, this ideal is $\Z[G]$ itself (see \Cref{subsection:primitivity} for a full discussion).

The parity condition concerns a certain mod-2 quadratic refinement $q$ of the intersection pairing which is studied in \Cref{section:q}. This is also a condition on the self-intersection of the orbit of $v$ which repairs the fact that the relative intersection pairing is blind to the $2$-torsion elements in $G$. Finally the coherence condition is an extra condition that is only relevant when the class $v$ has nontrivial stabilizer $G_v$. It asserts that the stabilizer subgroup be ``recoverable'' from the behavior of the element itself in relation to the classifying map of the cover - see \Cref{subsection:coherence} for details.

\begin{remark}[Where is stabilization necessary?]
Note that both the necessary and the sufficient conditions require the surface to be stabilized. In the necessary case, the conditions of isotropy, parity, and coherence hold without further stabilization, but primitivity is not guaranteed (and indeed can fail) without stabilization - see \Cref{lemma:primitive}. The proof of sufficiency makes systematic use of stabilization. At root, stabilization is used to repair points of self-crossing on $Y$ by rerouting the crossing through a new handle - see \Cref{section:resolution} - and our arguments provide no bound on the size of a stabilization required to realize even a single class. It is possible to refine the techniques described in the paper and obtain a result asserting that there is a fixed stabilization $X^+$ of $X$ (adding genus slightly more than the minimum number of generators of $G$) on which all classes in $H_1(X;\Z)$ satisfying \Cref{maintheorem:necessary} are relatively geometric. However, some of the required constructions are rather elaborate and so we postpone this line of inquiry to future work.
\end{remark}

\para{Relationship with surgery theory} As pointed out to us by Stephan Stolz, the central arguments in this paper have a close spiritual analog in the foundations of surgery theory. In that setting, one is interested in knowing when an immersion of a half-dimensional sphere $i: S^k \to M^{2k}$ can be promoted to an embedded sphere in the same homotopy class. The obstruction for doing so is encoded in an intersection form valued in $\Z[\pi_1(M)]$: it is necessary and sufficient that the corresponding homology class be (in the language of \Cref{maintheorem:necessary}) isotropic and even. Thus the results of this paper can be viewed as a kind of ``equivariant surgery theory'' in dimension $2$, replacing homotopy with homology.

There are some places where these storylines diverge. Firstly, we find that isotropy and parity are {\em not} sufficient to characterize the classes we are interested in. More subtle is the fact that we are forced to tailor our quadratic refinement $q$ in the parity condition to be sensitive to the stabilizer subgroup of the class $v \in H_1(X;\Z)$ under study; without this modification, the ``na\"ive'' construction of $q$ imported directly from surgery theory would not be sensitive enough to detect all points of self-intersection as required. This is the reason we do not package the relative intersection form and its quadratic refinement together into a single invariant as is typical - we would be forced to make our construction conscious of the stabilizer subgroups of the elements, which we feel would be an encumbrance that would obscure the overall picture.

\para{Relationship with Hermitian $K$-theory} The study of modules equipped with (skew)-Hermitian forms (such as the $\Z[G]$-module $H_1(X;\Z)$ equipped with $\pair{\cdot, \cdot}$) belongs to the domain of algebraic $K$-theory. The subject provides tools to study generating sets for the unitary automorphisms of the module, orbits of vectors, and other aspects of these groups which are quite pertinent to the problem at hand. While the ideas of this field served as a deep source of inspiration for this project, we should emphasize the fact that the arguments of this paper are purely topological and make no actual use of the technology of Hermitian $K$-theory. Indeed, we are able to obtain our results for arbitrary finitely-generated groups, whereas (to the author's knowledge) understanding of Hermitian $K$-theory for the group rings of finitely-generated groups is quite incomplete. We think it would be very worthwhile to investigate the extent to which the topological ideas of the paper could provide a new set of tools to study automorphism groups of skew-Hermitian modules by viewing them as quotients of the associated ``liftable'' subgroup of the mapping class group.

\para{Organization} \Cref{section:coverings} recalls the necessary background from the theory of covering spaces, and fixes our definitions and conventions regarding stabilizations. \Cref{section:pairing} establishes the basic theory of the relative intersection pairing appearing in the isotropy condition of \Cref{maintheorem:necessary}, and \Cref{section:q} does likewise for the quadratic refinement $q$ of the parity condition. \Cref{section:necessary} discusses the remaining necessary conditions of primitivity and coherence, and proves \Cref{maintheorem:necessary}. The final three sections are devoted to the proof of \Cref{maintheorem:sufficient}: \Cref{section:PU} discusses a special class of elements (``purely-unital vectors'') for which the question of relative geometricity can be resolved by hand, \Cref{section:resolution} introduces the {\em resolution process} underlying the main argument, and finally \Cref{section:proof} establishes \Cref{maintheorem:sufficient}.

\para{Acknowledgements} I would like to extend a hearty thanks to Corey Bregman for many profitable discussions on matters surrounding this work. I would also like to thank Stephan Stolz for highlighting the connections with surgery theory and for drawing my attention to the reference \cite{luck}, and to Sebastian Hensel for alerting me to the work of Flamm \cite{flamm}.

\section{Coverings of surfaces and their stabilizations}\label{section:coverings}

\subsection{Covering spaces and elevations of curves} Here we recollect some basic notions from covering space theory. The discussion here is largely routine and is included to fix notation and establish conventions, although \Cref{lemma:cyclicstab} is (slightly) less elementary, and will play an important role throughout the paper.

\para{Standing assumptions} Throughout the paper, $G$ denotes a finitely-generated group, and $f: X \to Y$ denotes a regular $G$-covering of a connected oriented surface $Y$ of finite type and nonempty boundary. We assume that $X$ is connected, so that $f$ is classified by a surjective homomorphism $\phi: \pi_1(Y) \to G$. We further assume that $Y$ has a distinguished boundary component $\Delta_0 \subset \partial Y$, subject to the condition that $\phi(\Delta_0) = 1$, when $\Delta_0$ is viewed as an element of $\pi_1(Y)$.

\para{Elevations} Let $\gamma \subset Y$ be a curve. An {\em elevation} of $\gamma$ is a choice of component $\tilde \gamma \subset f^{-1}(\gamma)$. We will use the notation
\begin{equation}\label{bullet}
\gamma_\bullet \in \pi_1(Y,*)
\end{equation}
to indicate the element of $\pi_1(Y,*)$ obtained by choosing an arbitrary basepoint $* \in \gamma$. Covering space theory asserts that the {\em conjugacy class} of $\gamma_\bullet$ is well-defined independently of choice of basepoint.

The deck group $G$ acts on the set of elevations of $\gamma$ from the {\em left} with cyclic stabilizer subgroup $\pair{\phi(\gamma_\bullet)}$, well-defined as a {\em subgroup} relative to a fixed basepoint, and as a {\em conjugacy class of subgroup} when no basepoint is specified. 

\para{Basepoint conventions} We will occasionally need to be careful about basepoints. We assume throughout that $Y$ is equipped with a basepoint $*$ contained on the distinguished boundary component $\Delta_0$. By the assumption that $\phi(\Delta_0) = 1$, the elevations of $\Delta_0$ are in bijective correspondence with $G$. We choose a distinguished elevation $\tilde \Delta_0$ of $\Delta_0$, and base $X$ at the lift $\tilde{*}$ of $*$ contained in $\tilde \Delta_0$.

\para{Geometric stabilizers are finite cyclic} The following lemma will play an important background role in what is to follow; it asserts that even when $G$ is infinite, the stabilizer subgroups of classes $v \in H_1(X;\Z)$ with geometric representatives are finite cyclic. A stronger version valid for finite covers appears as \cite[Proposition 2.1]{FH}. 

\begin{lemma}
\label{lemma:cyclicstab}
Let $G$ be a finitely-generated group and let $f: X \to Y$ be a regular $G$-covering. Suppose that $X$ is connected and has non-empty boundary. Let $v \in H_1(X;\Z)$ be represented by an elevation $\tilde \gamma$ of a nonseparating simple closed curve $\gamma \subset Y$. Then the stabilizer subgroup $G_v \leqslant G$ of $v$ is a finite cyclic group.
\end{lemma}
\begin{proof}
We first observe that necessarily $\abs{\phi(\gamma_\bullet)}$ is finite, otherwise the elevations of $\gamma$ do not have compact support and so do not represent elements $v \in H_1(X;\Z)$. Certainly $g \in \pair{\phi(\gamma_\bullet)}$ fixes the elevation $\tilde{\gamma}$ as an oriented simple closed curve, and so $G_v$ contains the finite cyclic group $\pair{\phi(\gamma_\bullet)}$. Any $g \not \in \pair{\phi(\gamma_\bullet)}$ takes $\tilde \gamma$ to some disjoint elevation $\tilde \gamma '$ of $\gamma$. It therefore suffices to show that distinct elevations lie in distinct homology classes.

Suppose to the contrary: then $\tilde \gamma$ and $\tilde \gamma '$ are disjoint homologous curves in $X$. Therefore $X \setminus \{\tilde \gamma, \tilde \gamma '\}$ is disconnected, and moreover at least one component is a {\em compact} subsurface $S$ with boundary $\tilde \gamma \cup \tilde \gamma '$. Necessarily then $S$ does not contain any elevation of any component of $\partial X$ and so $f(S)$ does not contain any component of $\partial Y$. But this is absurd: let $\alpha \subset Y$ be an arc connecting $\gamma$ to $\partial Y$ whose interior is disjoint from $\gamma$. There is a lift of $\alpha$ to $X$ that is contained in $S$ for sufficiently small $t$ (relative to an arbitrary parameterization beginning at $\tilde \gamma$), and since $\alpha$ does not cross $\gamma$, this lift of $\alpha$ never leaves $S$, showing that $S$ contains a component of $\partial X$, contrary to assumption.
\end{proof}

\para{Homology classes and their representatives: standing conventions} In the sequel we will frequently pass between a homology class $v \in H_1(X;\Z)$ and a representative cycle. Here we fix conventions that are to be understood throughout.

\begin{convention}
\label{standingconventions}
Let $v \in H_1(X;\Z)$ be arbitrary. By a {\em representative} for $v$, we mean an oriented weighted multicurve $\gamma \subset X$ with $[\gamma] = v$. By standard transversality considerations, we always assume that $\gamma$ is in {\em general position} with respect to the projection map $f: X \to Y$ - we assume that $f(\gamma)$ is immersed with a finite number of transverse self-intersection points and no triple intersections. If we consider two classes $v,w$ simultaneously, we moreover assume that $f(v)$ and $f(w)$ are transverse.   
\end{convention}

\subsection{Stabilization}
In this paper we do not work with the absolutely most general notion of stabilization one could formulate. To delineate our conventions and fix notation, we give a precise definition of the stabilization operation we consider. 
\begin{definition}[(Simple) stabilization]
\label{definition:stabilization}
Let $f: X \to Y$ be a regular $G$-cover satisfying the standing assumptions. A {\em stabilization} of $f$ is an embedding of regular $G$-covers
\[
\xymatrix{
X \ar[r]^{\tilde i} \ar[d]_f & X^+ \ar[d]^{f^+}\\
Y \ar[r]_i        & Y^+
}
\]
satisfying the following conditions:
\begin{enumerate}
    \item The total space $X^+$ is connected, so that $f^+: X^+ \to Y^+$ is classified by an extension $\phi^+: \pi_1(Y^+) \to G$ of $\phi$ (i.e. the restriction of $\phi^+$ to $\pi_1(Y)$ is given by $\phi$),
    \item The complement $Y^+ \setminus Y$ is a connected surface with two boundary components, one of which is $\Delta_0$ and the other of which (denoted $\Delta_1$) is contained in $\partial Y^+$, 
    \item $\phi^+(\Delta_1) = 1$ and we take $\Delta_1$ as the distinguished boundary component of $Y^+$.
\end{enumerate}
We write 
\[
H_1(X^+;\Z)^{stab} \leqslant H_1(X^+;\Z)
\]
to denote the submodule of $H_1(X^+;\Z)$ induced by the inclusion $X^+ \setminus X \into X^+$.

A stabilization $f^+$ of $f$ is {\em simple} if the restriction of $\phi^+$ to $\pi_1(Y^+ \setminus Y)$ is the trivial homomorphism. 
\end{definition}
Note that while our conventions provide for a canonical choice of new distinguished boundary component $\Delta_1$ of $Y^+$, this does not in general lift to a canonical choice of elevation $\tilde \Delta_1$. Such a choice can be made as follows: let $\alpha \subset Y^+ \setminus Y$ be a properly-embedded arc connecting $\Delta_0$ to $\Delta_1$; then the lift $\tilde \alpha$ based at $\tilde \Delta_0$ ends at an elevation of $\Delta_1$ which we take to be the distinguished lift $\tilde \Delta_1$. If $f^+$ is a {\em simple} stabilization, then this choice is independent of such $\alpha$ and hence is canonical, but in general the set of possible choices is in bijection with elements of the subgroup $\phi(\pi_1(Y^+ \setminus Y)) \leqslant G$. 

\subsection{Basic $g$-handles}
In practice, the stabilizations we consider will be of an especially simple form. 
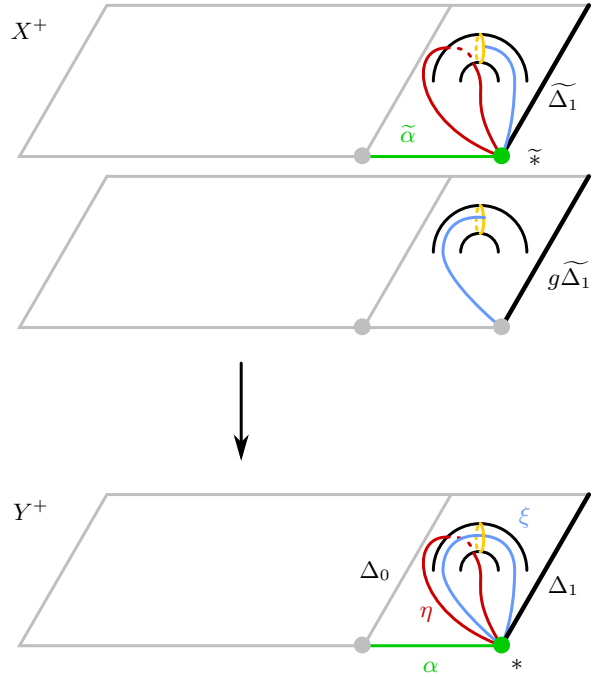
\begin{figure}[ht]
\labellist
\small
\pinlabel $\Delta_0$ [br] at 143 25.51
\pinlabel $\Delta_1$ [tl] at 198 31.18
\pinlabel $\tilde{\Delta_1}$ [tl] at 198 217
\pinlabel $g\tilde{\Delta_1}$ [tl] at 198 150
\pinlabel $*$ [tl] at 184.07 0.00
\pinlabel $\tilde{*}$ [l] at 189.90 187.07
\pinlabel \textcolor{mygreen}{$\alpha$} [t] at 155.89 0.00
\pinlabel \textcolor{mygreen}{$\tilde{\alpha}$} [b] at 147.39 189.90
\pinlabel \blue{$\xi$} [bl] at 187.07 45.18
\pinlabel \red{$\eta$} [tr] at 158.89 19.84
\pinlabel $Y^+$ [br] at 14.17 48.18
\pinlabel $X^+$ [br] at 14.17 229.58
\endlabellist
\includegraphics[scale=1]{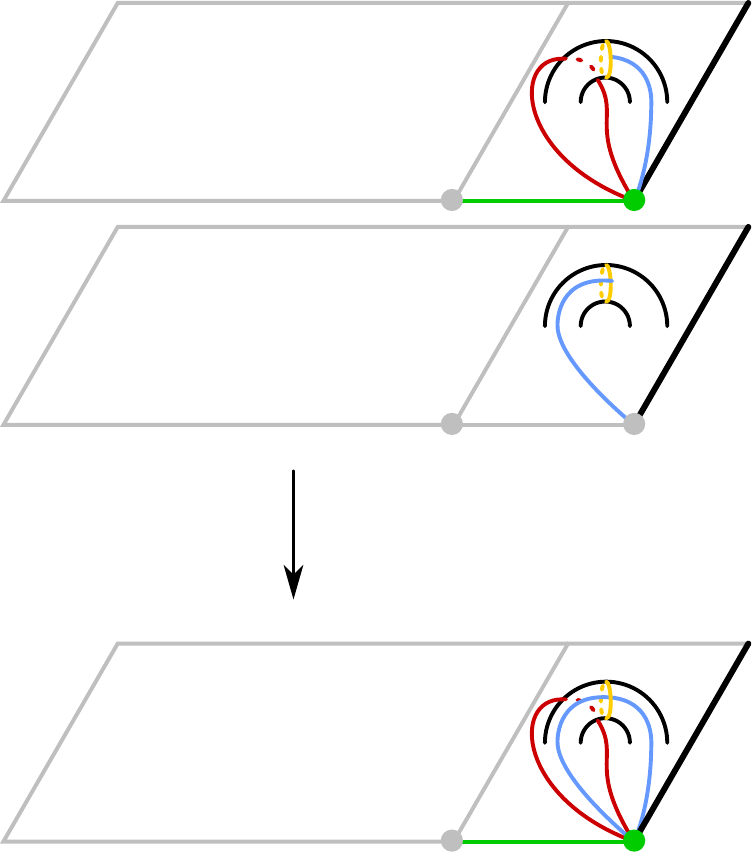}
\caption{A basic $g$-handle lying below its preimage in $X^+$. The curve $\eta$ lifts to $X^+$, while the lift of $\xi$ based at $\tilde{*}$ ends at $g \tilde{*}$. More generally, crossing over the branch cut from the right moves from sheet $h$ to sheet $hg$.}
\label{figure:basicg}
\end{figure}
\begin{definition}[Basic $g$-handle]
\label{def:basicg}
Let $g \in G$ be a chosen element. A {\em basic $g$-handle} is a torus with two boundary components $H$, along with a homomorphism $\phi: \pi_1(H) \to G$ given by
\[
\phi(\xi) = g, \quad \phi(\eta) = 1, \quad \phi(\Delta_0) = 1, \quad \phi(\Delta_1) = 1, 
\]
(where the elements $\xi, \eta, \Delta_0, \Delta_1 \in \pi_1(H)$ are shown in \Cref{figure:basicg}), as well as a choice of properly-embedded arc $\alpha$ connecting $\Delta_0$ to $\Delta_1$ (also shown in \Cref{figure:basicg}).

When $f^+: X^+ \to Y^+$ is a stabilization of $f$ such that there is a homeomorphism $Y^+ \setminus Y \cong H$ with a basic $g$-handle $H$ that is compatible with the homomorphisms to $G$, we say that $f^+$ is obtained from $f$ by {\em attaching a basic $g$-handle}.
\end{definition}
Note that the possibility that $g = 1 \in G$ is allowed; we hope that the term ``basic $1$-handle'' should not create too much confusion with the more conventional Morse-theoretic meaning, since stabilizing by attaching a basic $g$-handle to $Y$ for any $g \in G$ is, on a topological level, merely attaching a $1$-handle to $Y$ in the Morse-theoretic sense. Note also that a simple stabilization is merely a sequence of stabilizations by basic $1$-handles.

\section{The relative intersection pairing}\label{section:pairing}
In this section, we discuss a crucial algebraic invariant of $H_1(X;\Z)$ which we call the {\em relative intersection pairing}. This has appeared in the literature in various guises and by various names, in surface topology, $3$-manifold topology, and in surgery theory. See e.g. \cite{putman,luck} for further discussion and references. 

\Cref{subsection:algebraic} establishes the basic algebraic properties of the relative intersection pairing. In \Cref{subsection:localformula}, we establish a {\em local formula} for the relative intersection pairing in terms of intersection points of the projections of cycles on $Y$. Finally in \Cref{subsection:simplestab}, we study how a simple stabilization affects the homology of a cover and the associated relative intersection form.

\subsection{Basic algebraic properties}\label{subsection:algebraic}
We denote the ordinary algebraic intersection pairing on $H_1(X;\Z)$ by $( \cdot, \cdot )$. This is a bilinear alternating form valued in $\Z$. Using this, we define the relative intersection pairing valued in $\Z[G]$.
\begin{definition}\label{def:relint}
The {\em relative intersection pairing} is the form
\[
\pair{\cdot, \cdot}: H_1(X; \Z) \otimes H_1(X; \Z) \to \Z[G]
\]
defined by the formula
\begin{equation}\label{equation:relint}
\pair{v,w} = \sum_{g \in G} (v,gw) g.
\end{equation}
\end{definition}
\begin{remark}
When $G$ is infinite, the sum defining $\pair{\cdot, \cdot}$ is over an infinite set, posing the question of well-definedness. However, the classes $v, w$ are necessarily represented by cycles with compact support, while the covering group $G$ acts freely and properly-discontinuously, so that $(v,gw)$ is nonzero for only a finite number of elements $g \in G$. 
\end{remark}

\begin{lemma}\label{lemma:skewherm}
The relative intersection form is {\em skew-Hermitian}, i.e. $\pair{\cdot, \cdot}$ is $\Z[G]$-linear in the first component, and satisfies $\pair{w,v} = -\overline{\pair{v,w}}$, where $\overline{\cdot}: \Z[G] \to \Z[G]$ is the involution induced by the inversion map on $G$.
\end{lemma}
\begin{proof}
This is a straightforward exercise and is left to the reader. 
\end{proof}

\subsection{The local formula}\label{subsection:localformula} 
The goal of this subsection is to establish \Cref{lemma:localformula}, which gives a formula for $\pair{v,w}$ ``localized'' over the points of intersection of suitable representatives for $v,w$ projected onto $Y$. This is formulated in terms of a {\em local sheet index} presented in \Cref{definition:localindex}.

Suppose $v,w \in H_1(X;\Z)$ are represented by oriented cycles $\gamma, \delta$ on $X$. We do not assume that $v$ and $w$ are relatively geometric, but we do assume that $\gamma$ is stabilized setwise by some {\em finite cyclic} subgroup $G_\gamma$ and likewise that $\delta$ is stabilized by $G_\delta$ (of course if $v,w$ are relatively geometric and nonseparating, then this is forced by \Cref{lemma:cyclicstab}). Observe that $G_\gamma$ is a subgroup of the stabilizer $G_v$ of $v = [\gamma]$ and likewise that $G_\delta \leqslant G_w$. For now we do not assume that this is an equality (but see \Cref{lemma:eqvtrep}). 

By \Cref{standingconventions}, $f(\gamma), f(\delta)$ have a finite number of crossings. Each crossing contributes a local factor to the pairing $\pair{v,w}$; \Cref{lemma:localformula} below records this formula. To state it, we define the {\em local sheet index} as follows. Let $p_i \in Y$ be a point of crossing between $f(\gamma), f(\delta)$. In $X$, the local branch of $f(\gamma)$ is covered by an orbit $G_\gamma \gamma$ of local branches of $\gamma$, and likewise $f(\delta)$ is covered by an orbit $G_\delta \delta$ of local branches of $\delta$. 

The sheets above $p_i$ can be non-canonically identified with $G$ via the following procedure: choose a basepoint $* \in Y$, and choose a distinguished lift of $*$ in $X$; in this way the fiber above $*$ is identified with $G$. Next choose a path $\alpha \subset Y$ connecting $*$ to $p_i$, and use the lifts of $\alpha$ to identify the fiber above $p_i$ with $G$. If some other path $\beta \subset Y$ is used instead, covering space theory shows that the two identifications differ by {\em right-}multiplication by some element $k \in G$. Under any such identification, the local branches of $\gamma$ lie in the sheets corresponding to some coset $G_\gamma g_1$, and likewise the local branches of $\delta$ lie in the sheets corresponding to a coset $G_\delta g_2$. 

\begin{definition}[Local sheet index]\label{definition:localindex}
With notation as in the above paragraph, the {\em local sheet index at $p_i$} is the double coset 
\[
i(\gamma, \delta, p_i) := G_\gamma g_1 g_2^{-1} G_\delta.
\]
\end{definition}

\begin{lemma}
\label{lemma:sheetindexWD}
The local sheet index at $p_i$ is well-defined independently of the identification of sheets above $p_i$ with $G$.
\end{lemma}

\begin{proof}
As mentioned above, if $m_j: f^{-1}(p_i) \to G$ for $j = 1,2$ are two such markings, then there exists $k \in G$ such that 
\[
m_2(q) = m_1(q)k
\]
for all $q \in f^{-1}(p_i)$. If the distinguished cosets under the first marking are $G_\gamma g_1, G_\delta g_2$, then in the second marking they are $G_\gamma g_1 k, G_\delta g_2 k$. The double coset is $G_\gamma g_1 g_2^{-1} G_\delta$ in either case.
\end{proof}

In order to understand the contribution to $\pair{v,w}$ associated to a crossing point with given local sheet index, we introduce a key piece of notation. For an element $v \in H_1(X;\Z)$ {\em with finite cyclic stabilizer $G_v \leqslant G$}, define the element $\zeta_v \in \Z[G]$ via
\begin{equation}\label{zetadef}
\zeta_v := \sum_{g \in G_v} g.
\end{equation}
With $\zeta_v$ defined, we come to the formulation of the {\em local crossing factor}.
\begin{definition}[Local crossing factor]
\label{definition:localfactor}
With notation as in \Cref{definition:localindex}, the {\em local crossing factor at $p_i$} is the element $c(\gamma, \delta, p_i) \in \Z[G]$ defined by
\[
c(\gamma, \delta, p_i) := \zeta_\gamma (g_1 g_2^{-1}) \zeta_\delta
\]
(note that this expression indeed depends only on the associated double coset). 
\end{definition}

\begin{lemma}\label{lemma:localformula}
Let $v,w \in H_1(X;\Z)$ be represented by oriented cycles $\gamma, \delta$ on $X$. Suppose that $f(\gamma), f(\delta)$ are immersed in $Y$ and intersect in general position at points $p_1, \dots, p_k \in Y$. For a point of intersection $p_i$, let $\epsilon_i \in \{\pm 1\}$ denote the local intersection number, and let $c(\gamma, \delta, p_i) = \zeta_\gamma (g_{1,i} g_{2,i}^{-1}) \zeta_\delta \in \Z[G]$ denote the local crossing factor. Then 
\[
\pair{v,w} = \pair{[\gamma], [\delta]} = \sum_{i =1}^k \epsilon_i c(\gamma, \delta, p_i) = \zeta_\gamma \left(\sum_{i = 1}^k \epsilon_i (g_{1,i} g_{2,i}^{-1}) \right) \zeta_\delta.
\]
\end{lemma}
\begin{proof}
By \eqref{equation:relint}, $\pair{v,w}$ is computed as follows:
\[
\pair{v,w} = \pair{[\gamma], [\delta]} = \sum_{g \in G} (\gamma, g \delta) g.
\]
For fixed $g \in G$, the intersection number $(\gamma, g \delta)$ has a local formula given by summing the local intersection numbers $\epsilon_i$ at the crossings of $\gamma$ with $g \delta$. Such crossings appear only at points in the fibers $f^{-1}(p_i)$. Let $c_{i,g}$ denote the number of crossings between $\gamma$ and $g \delta$ that occur in the fiber $f^{-1}(p_i)$. Then
\[
\sum_{g \in G} (\gamma, g \delta) g = \sum_{g \in G}\sum_{i= 1}^k \epsilon_i c_{i,g} g = \sum_{i =1}^k \epsilon_i \sum_{g \in G} c_{i,g}g.
\]
Recall that we have identified the fiber $f^{-1}(p_i)$ with $G$ in such a way that the local branches of $\gamma$ lie in sheets corresponding to some coset $G_\gamma g_{1,i}$, and the local branches of $\delta$ lie in sheets corresponding to a coset $G_\delta g_{2,i}$. Then the local branches of $g \delta$ lie in the sheets $g G_\delta g_{2,i}$, and so $c_{i,g}$ can be computed by the following expression:
\[
c_{i,g} = \abs{G_\gamma g_{1,i} \cap g G_\delta g_{2,i}} = \abs{G_\gamma (g_{1,i} g_{2,i}^{-1}) \cap g G_\delta}.
\]
To finish the argument, we claim that 
\[
\sum_{g \in G} \abs{G_\gamma (g_{1,i} g_{2,i}^{-1}) \cap g G_\delta} g = \left( \sum_{f \in G_\gamma} f \right) g_{1,i}g_{2,i}^{-1} \left( \sum_{h \in G_\delta} h \right) = \zeta_\gamma (g_{1,i} g_{2,i}^{-1}) \zeta_h = c(\gamma, \delta, p_i).
\]
To see this, we expand the double summation:
\[
\left( \sum_{f \in G_\gamma} f \right) (g_{1,i} g_{2,i}^{-1}) \left( \sum_{h \in G_\delta} h \right) = \sum_{(f,h) \in G_\gamma \times G_\delta} f g_i h.
\]
Thus $g \in G$ appears in this sum as many times as $g$ admits an expression of the form $g = f(g_{1,i}g_{2,i}^{-1}) h$ for some $f \in G_\gamma, h \in G_\delta$, or equivalently, the number of equalities $f (g_{1,i}g_{2,i}^{-1}) = g h^{-1}$, i.e. the cardinality $\abs{G_\gamma (g_{1,i}g_{2,i}^{-1}) \cap g G_\delta} = c_{i,g}$.
\end{proof}

In practice, \Cref{lemma:localformula} can be applied as follows: tracing along $f(\gamma) \subset Y$, at each crossing $p_i$ with $f(\delta)$, add $\epsilon_i c(\gamma, \delta, p_i)$, where $\epsilon_i = 1$ if and only if the local branches of $\gamma, \delta$ {\em in that order} are positively-oriented. As a corollary of this point of view, we examine how the formula specializes in the case of self-intersection.
\begin{corollary}\label{corollary:selfint}
Let $v \in H_1(X;\Z)$ be represented by an oriented cycle $\gamma$ subject to \Cref{standingconventions}, self-intersecting at points $p_1, \dots, p_k \in Y$. Suppose that the local sheet index at $p_i$ is given by the double coset $G_v g_i G_v$ when the local orientation of the branches is positive. Then
\[
\pair{v,v} = \zeta_\gamma \sum_{i = 1}^k (g_i - g_i^{-1}) \zeta_\gamma.
\]
\end{corollary}
\begin{proof}
As one traces along $\gamma$, each self-intersection point is traversed {\em twice}, once with positive local orientation and local crossing factor $\zeta_\gamma g_i \zeta_\gamma$, and once with negative local orientation and crossing factor $\zeta_\gamma g_i^{-1} \zeta_\gamma$.
\end{proof}

\subsection{Homology of simple stabilizations}\label{subsection:simplestab}
Having defined the relative intersection form, we record here some basic information on the homological effect of a simple stabilization.

\begin{proposition}
\label{proposition:simplehomol}
Let $f^+: X^+ \to Y^+$ be a simple stabilization of $f: X \to Y$ obtained by adding a basic $1$-handle. Then 
\begin{equation}\label{equation:H1Xplus}
H_1(X^+;\Z) \cong H_1(X;\Z) \oplus \Z[G]\pair{x,y} 
\end{equation}
with $\Z[G] \pair{x,y}$ denoting a free $\Z[G]$-module of rank $2$. Under the relative intersection form, $H_1(X;\Z)$ is orthogonal to $\Z[G]\pair{x,y}$, and the restriction of $\pair{\cdot, \cdot}$ to $\Z[G]\pair{x,y}$ is {\em hyperbolic}: \[
\pair{x,y} = 1,\quad \pair{x,x} = \pair{y,y} = 0.
\]
The classes $x,y$ are represented by the {\em based} loops shown in \Cref{figure:simplehomol}. 
\end{proposition}
\begin{proof}
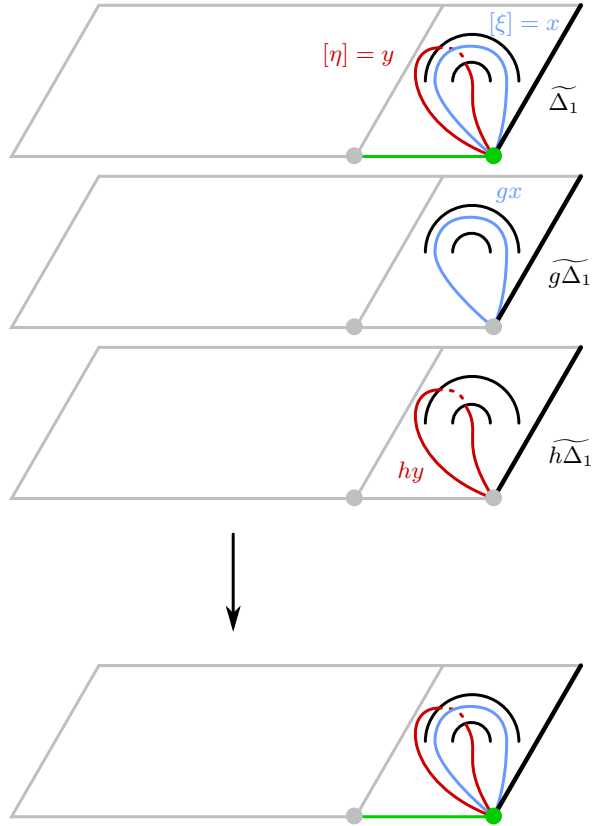
\begin{figure}[ht]
\labellist
\small
\pinlabel $\blue{[\xi]=x}$ [bl] at 178.57 295.44
\pinlabel $\red{[\eta]=y}$ [br] at 147.39 283.44
\pinlabel $\tilde{\Delta_1}$ [tl] at 201.24 280.60
\pinlabel $\tilde{g\Delta_1}$ [tl] at 201.24 215.41
\pinlabel $\tilde{h\Delta_1}$ [tl] at 201.24 147.39
\pinlabel \blue{$gx$} [bl] at 181.40 232.42
\pinlabel \red{$hy$} [tr] at 158.73 138.88
\endlabellist
\includegraphics[scale=1]{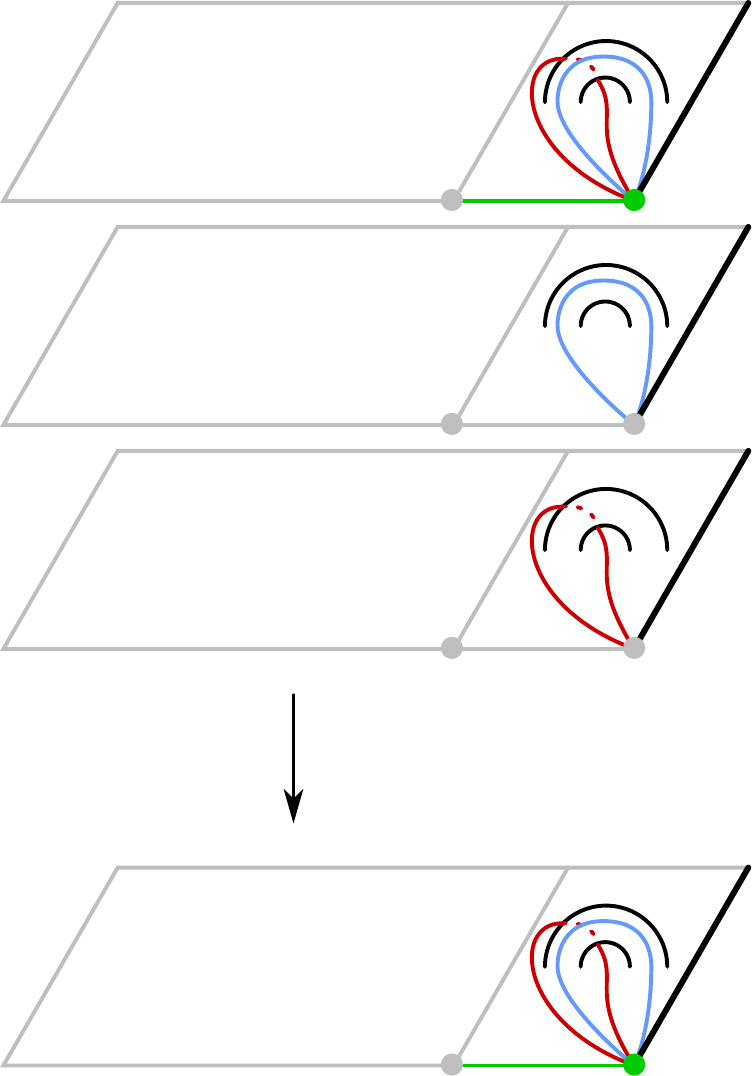}
\caption{Coordinates on the homology of a stable handle.}
\label{figure:simplehomol}
\end{figure}
The complement $X^+ \setminus X$ is a disjoint union of $\abs{G}$ surfaces, each one homeomorphic to a torus with two boundary components. By the Mayer-Vietoris sequence, the isomorphism \eqref{equation:H1Xplus} holds on the level of abelian groups. The action of $G$ on $X^+ \setminus X$ is given by permuting the components, and it follows that \eqref{equation:H1Xplus} holds as an isomorphism of $\Z[G]$-modules. The assertions concerning the relative intersection form follow from the geometric formula for $\pair{\cdot, \cdot}$ given in \Cref{lemma:localformula} in combination with the geometric construction of cycles representing $x,y$ shown in \Cref{figure:simplehomol}, both of which are supported on $X^+ \setminus X$.
\end{proof}

\section{Parity: constructing a quadratic refinement of $\pair{\cdot, \cdot}$}\label{section:q}

Recall from \Cref{corollary:selfint} that a self-intersection of $\gamma \subset Y$ at some point $p\in Y$ contributes a local factor of $\zeta_\gamma(g - g^{-1})\zeta_\gamma$ to $\pair{[\gamma], [\gamma]}$, where $g \in G$ measures the difference in sheets between the local branches appearing above $p \in \gamma$. Note that when $g^2 = 1$, there is no contribution to $\pair{[\gamma], [\gamma]}$: {\em self-crossings with $2$-torsion sheet differences are invisible to the relative intersection pairing.} Thus, isotropy alone is not sufficient to ensure that a class $v \in H_1(X;\Z)$ admits a representative with no self-intersections. The goal of this section is to describe here a certain quadratic refinement of the mod-$2$ relative intersection pairing which can detect such self-crossings; this is given in \Cref{definition:q}.\\

Like the relative intersection form, the quadratic refinement counts self-intersection points of a curve $\gamma$ under $f$. Achieving the sharpest possible count will require a brief detour into group theory, encapsulated in \Cref{lemma:scgprops}. Suppose that $\gamma \subset X$ is a multicurve that is setwise fixed by some finite cyclic subgroup $C \leqslant G$, and suppose that $f(\gamma)$ has a transverse double point at $p \subset Y$. The local sheet index at $p$ is then some double coset of the form $C h C$. Now let $g \in G$ be an element of order $2$. There is an action of $C \times C$ on $G$ where the first factor acts by left multiplication, and the second by right multiplication by the inverse. Define
\[
S_{C,g}: = \Stab_{C \times C}(g)
\]
as the stabilizer of $g$ under this action (note that the orbit is the double coset $C g C$). 

\begin{lemma}\label{lemma:scgprops}
Let $S_{C,g} \leqslant C \times C$ be defined as above.
\begin{enumerate}
    \item $S_{C,g}$ is {\em graph-like}: $S_{C,g} \cap (C \times \{1\}) = \{(1,1)\}$, and hence the projection $S_{C,g} \to C$ along the second factor is an injection,
    \item $S_{C,g}$ is {\em symmetric}: $(c,d) \in S_{C,g}$ if and only if $(d,c) \in S_{C,g}$. Thus, letting $\overline{S_{C,g}}$ denote the embedding of $S_{C,g}$ into $C$ via projection, the factor-swapping involution $(c,d) \mapsto (d,c)$ on $S_{C,g}$ descends to an involution $\iota(c) = d$ on $\overline{S_{C,g}}$.
    \item $\overline{S_{C,g}}$ is normalized by $g$, and conjugation by $g$ induces the involution $\iota$,
    \item The subgroup of $G$ generated by $\overline{S_{C,g}}$ and $g$ has the structure of a semi-direct product $\tilde{S_{C,g}}:= \pair{g} \ltimes \overline{S_{C,g}}$, with $g$ acting on $\overline{S_{C,g}}$ via $\iota$. 
\end{enumerate}
\end{lemma}
\begin{proof}
For (1), observe that if $(c,1) \in S_{C,g}$, then $c g = g$, i.e. $c = 1$. For (2), we suppose that $(c,d) \in S_{C,g}$, so that there is an equation of the form $c g d^{-1} = g$. Inverting both sides and recalling that $g^{-1} = g$ establishes (2). 

For (3), let $c \in \overline{S_{C,g}}$ be induced from the element $(c,d) \in S_{C,g}$. Then there is an expression $c g d^{-1} = g$, which rearranges to $g^{-1} c g = d$. For (4), we observe that any word $w \in \pair{\overline{S_{C,g}}, \pair{g}}$ has an expression involving at most one $g$, since any subword of the form $g c g$ can be re-written as $\iota(c)$. The semi-direct product structure follows easily.
\end{proof}

\begin{lemma}\label{lemma:evenints}
Let $\gamma \subset X$ be a multicurve such that $\gamma$ is setwise fixed by some finite cyclic subgroup $C \leqslant G$, and let $g \in G$ be of order $2$. Then $\gamma \cap g \gamma \subset X$ decomposes as a union of $\tilde{S_{C,g}}$-orbits and hence $\abs{\gamma \cap g \gamma}$ is divisible by $\abs{\tilde{S_{C,g}}} = 2\abs{S_{C,g}}$.
\end{lemma}
\begin{proof}
Suppose that $p \in \gamma \cap g \gamma$, and take $g^\epsilon c \in \tilde{S_{C,g}}$ for $\epsilon \in \{0,1\}$. Then
\[
g^\epsilon c p \in g^\epsilon c \gamma \cap g^\epsilon c g \gamma.
\]
Note that $c \gamma = \gamma$ since $\overline{S_{C,g}} \leqslant C$, so that $g^\epsilon c \gamma = g^\epsilon \gamma$. Also note that $c g = g d$ for some $d \in \overline{S_{C,g}}$, and so 
\[
g^\epsilon c g \gamma  = g^{\epsilon+1} d \gamma = g^{\epsilon+1} \gamma.
\]
Thus,
\[
g^\epsilon c p \in g^\epsilon \gamma \cap g^{\epsilon+ 1} \gamma = \gamma \cap g \gamma.
\]
The claimed divisibility follows from the fact that $G$ (and hence $\tilde{S_{C,g}} \leqslant G$) acts {\em freely} on $X$.
\end{proof}

We are now in a position to define the quadratic refinement. We begin in \Cref{definition:qgeom} with a geometric formulation in terms of a choice of representing cycle, and then prove in \Cref{lemma:parityhomol} that this is independent of choice.

\begin{definition}[Parity, geometric definition]\label{definition:qgeom}
Let $\gamma \subset X$ be a multicurve fixed setwise by some cyclic subgroup $C \leqslant G$, and let $G_2^*$ denote the set of elements of order $2$ in $G$. The {\em parity} of $\gamma$ is the vector $q(\gamma) \in (\Z/2\Z)^{G_2^*}$ where the entry indexed by $g \in G_2^*$ is given as 
\[
q_g(\gamma):= \frac{\abs{\gamma \cap g \gamma}}{\abs{\tilde{S_{C,g}}}} \pmod 2,
\]
i.e. the mod-$2$ count of $\tilde{S_{C,g}}$-orbits in the intersection.
\end{definition}

The objective is now to see that $q(\gamma)$ depends only on the homology class $[\gamma]$. To that end, we establish the following lemma.
\begin{lemma}
\label{lemma:eqvtrep}
Let $H\leqslant G$ be a finite subgroup, and suppose $v \in H_1(X;\Z)^H$ is an $H$-invariant class. Then $v$ admits a representative $v = [\gamma]$ with $\gamma \subset X$ an oriented multicurve that is fixed setwise by $H$. Moreover, if $\gamma'$ is another $H$-invariant representative, then $\gamma$ and $\gamma'$ are $H$-equivariantly cobordant: there is a properly embedded $H$-invariant subsurface $\Gamma \subset X \times [0,1]$ such that $\Gamma \cap (X \times \{0\}) = \gamma$ and $\Gamma \cap (X \times \{1\}) = \gamma'$. 
\end{lemma}
\begin{proof}
Set $X_H:= X/H$, and define the projection map $f_H: X \to X_H$. The theory of the transfer map provides for a homomorphism
\[
\tau: H_1(X_H;\Z) \to H_1(X;\Z)^H
\]
which is represented on the cycle level by $\tau([\gamma]) = [f^{-1}(\gamma)]$. The composition 
\[
f_{H,*} \tau: H_1(X_H;\Z) \to H_1(X_H;\Z)
\]
is given by multiplication by $\abs{H}$. As $H_1(X_H;\Z)$ is torsion-free, we can define an isomorphism
\[
\frac{1}{\abs{H}}: \abs{H}\ H_1(X_H;\Z) \to H_1(X_H;\Z).
\]
We then see that $\frac{1}{\abs{H}} f_{H,*}$ splits $\tau$, and so $\tau$ is a surjection. This implies the first claim, than any $v \in H_1(X;\Z)^H$ admits a representative by an $H$-invariant multicurve: represent $v = \tau([\bar \gamma])$ for suitable $[\bar \gamma] \in H_1(X_H;\Z)$. 

Now suppose that $\gamma, \gamma'$ are two $H$-invariant multicurve representatives of a class $v \in H_1(X;\Z)^H$. Then, as (unweighted) multicurves, $\gamma = f^{-1}(f(\gamma))$ and likewise $\gamma' = f^{-1}(f(\gamma'))$, and moreover $[f(\gamma)] = [f(\gamma')]$ as elements of $H_1(X_H;\Z)$. Since $K(\Z,1) = S^1$, Poincar\'e duality implies that $f(\gamma)$ (resp. $f(\gamma')$) can be represented as the level set $\sigma_i^{-1}(0)$ of a smooth map $\sigma_i: X_H \to S^1$ for $i = 0$ (resp. $i = 1$). Moreover, since $[f(\gamma)] = [f(\gamma')]$, the maps $\sigma_0$ and $\sigma_1$ are {\em homotopic} via some homotopy $\sigma_t, t \in [0,1]$. Standard transversality arguments then imply that $\sigma_t$ can be chosen so that $0$ is a regular value of $\sigma_t: X_H \times [0,1] \to S^1$, and so $M = \sigma_t^{-1}(0) \subset X_H \times [0,1]$ provides a cobordism between $f(\gamma)$ and $f(\gamma')$. Taking the preimage $f^{-1}(M)$ then gives an $H$-equivariant cobordism between $\gamma$ and $\gamma'$ in $X \times [0,1]$.
\end{proof}

\begin{lemma}\label{lemma:parityhomol}
The parity $q(\gamma)$ depends only on the class $[\gamma] \in H_1(X;\Z)$.
\end{lemma}
\begin{proof}
Let $C \leqslant G$ be a finite cyclic subgroup and let $\gamma, \gamma' \subset X$ be $C$-invariant oriented multicurves determining the same class in $H_1(X;\Z)^C$. By \Cref{lemma:eqvtrep}, there is some $C$-invariant properly-embedded subsurface $\Gamma\subset X \times I$ realizing a cobordism between $\gamma$ and $\gamma'$. Altering $\Gamma$ by a small $C$-equivariant perturbation if necessary, the intersection $M = \Gamma \cap g \Gamma$ is a properly-embedded $1$-submanifold of $X \times I$ invariant under the action of $\tilde{S_{C,g}} \leqslant G$. 

In this framework, the value $q_g(\gamma)$ is given as
\[
q_g(\gamma) = \frac{\abs{M \cap (X \times \{0\})}}{\abs{\tilde{S_{C,g}}}} \pmod 2;
\]
likewise
\[
q_g(\gamma') = \frac{\abs{M \cap (X \times \{1\})}}{\abs{\tilde{S_{C,g}}}} \pmod 2.
\]
To see these are the same, we analyze the components of $M$. As a compact $1$-manifold with boundary, $M$ decomposes as a finite number of circles and properly-embedded arcs. The circles do not intersect $X \times \partial I$ and so do not contribute to our analysis. The arcs of $M$ come in three types, depending on whether $0,1,$ or $2$ ends are embedded in $X \times \{0\}$; say an arc is of {\em type $i$} if it has $i$ ends in $X \times \{0\}$. The following claim is the central fact from which the lemma will follow.

\para{Claim} {\em The action of $\tilde{S_{C,g}}$ on the arcs of $M$ is a {\em free} type-preserving involution.}\\

Modulo the claim, we see how the result follows. If the arcs in some orbit $\tilde{S_{C,g}}A$ are of type $1$, then this orbit contributes $1$ to each of $q_g(\gamma), q_g(\gamma')$. If the arcs in the orbit are of type $0$ or $2$, then it contributes $2$ to one of $q_g(\gamma)$ or $q_g(\gamma')$ and $0$ to the other. It follows that $q_g(\gamma) = q_g(\gamma') \pmod 2$. 

We prove the claim. It is first of all clear that the action is type-preserving, since $G$ fixes each $X_t:= X \times \{t\}$. If $A \subset M$ is of type $1$, we consider $p = A \cap X_0$. Then for $h \in \tilde{S_{C,g}}$, we have $h A \cap X_0 = h p$, and $hp \ne p$ since $h$ acts freely on each level $X_t$. Hence $hA \ne A$. Suppose next that $A$ is of type $2$; let $A \cap X \times \{0\} = \{p,q\}$. Suppose that $hA = A$ for some $h \in \tilde{S_{C,g}}$. Then the set $\{p,q\} \subset X$ must be $h$-invariant, and since $h$ acts freely, this shows $h^2 = 1$. Consider the projection $p_2: A \to I$. By perturbing $\Gamma$ if necessary, we can assume that $p_2$ is a Morse function for $A$. Each critical point changes the Euler characteristic of the sublevel set by $1$, and since the sublevel set for small values of $t$ has Euler characteristic $2$ and $A$ itself has Euler characteristic $1$, it follows that there must be an odd number of critical points, and thus some value of $t$ for which $A \cap X_t$ has odd cardinality. On the other hand, the action of $h$ on $X_t$ is free. It follows that $A \cap X_t$ cannot be $h$-invariant, and hence $A$ itself is not fixed by $h$ as claimed. The same argument can be applied to $A$ of type $0$.
\end{proof}

Following \Cref{lemma:parityhomol}, we make the following definition.
\begin{definition}\label{definition:q}
The {\em parity} is the function $q: H_1(X;\Z) \to (\Z/2\Z)^{G_2^*}$ given by $q(v) = q(\gamma)$, where $\gamma$ is any representative of $v \pmod 2$ as a multicurve (we take $q(v) = \vec 0$ if $v = \vec 0 \pmod 2$).
\end{definition}

We record a local formula for $q$ analogous to \Cref{lemma:localformula}.
\begin{lemma}
\label{lemma:qlocal}
Let $v\in H_1(X;\Z)$ have finite cyclic stabilizer group $G_v$, and suppose that $v \pmod 2$ is represented by the $G_v$-invariant multicurve $\gamma \subset X$. Enumerate the self-intersection points of $f(\gamma)$ as $p_1, \dots, p_k \subset Y$, and suppose that $p_i$ has local sheet index $G_v g_i G_v$. Then for any $g \in G_2^*$,
\[
q_g(v) = (\# G_v g_i G_v = G_v g G_v) \pmod 2.
\]
\end{lemma}
\begin{proof}
By definition, 
\[
q_g(v) = \frac{\abs{\gamma \cap g \gamma}}{\abs{\tilde{S_{G_v,g}}}}.
\]
The intersection points all occur in the fibers $f^{-1}(p_i)$. Identify the local branches of $\gamma$ above some $p_i$ with the right cosets $G_v$ and $G_v g_i$. Then the intersection points of $\gamma$ with $g \gamma$ are of two (mutually-exclusive) types: a branch of $\gamma$ in $G_v$ intersecting a branch of $g \gamma$ in $g G_v g_i$, and a branch of $\gamma$ in $G_v g_i$ intersecting a branch of $g \gamma$ in $g G_v$. There are $\abs{G_v \cap g G_v g_i} = \abs{G_v g_i \cap g G_v}$ of each type. 

To proceed, we count $\abs{G_v g_i \cap g G_v}$, i.e. solutions $(h_1, h_2) \in G_v \times G_v$ to the equation $h_1 g_i =  g h_2$. If the double cosets $G_v g G_v$ and $G_v g_i G_v $ are not equal, there are no solutions, and hence no local contribution to $q_g(v)$. If $G_v g G_v = G_v g_i G_v $, then the orbit-stabilizer theorem implies that the solutions are in bijection with the stabilizer $S_{G_v,g}$ of $g$ under the action of $G_v \times G_v$. Thus in total, when $G_v g_i G_v = G_v g G_v$, there are $2 \abs{S_{G_v,g}} = \abs{\tilde{S_{G_v,g}}}$ intersections in the fiber above $p_i$, and this contributes $1$ to the value $q_g(v)$.
\end{proof}

For later use, we record a crucial structural property of $q$. 
\begin{lemma}\label{lemma:qprops}
Fix $g \in G_2^*$ and $x,y \in H_1(X;\Z)$ such that $G_x = G_y = G_{x+y} = \pair{1}$. The function $q_g: H_1(X;\Z) \to \Z/2\Z$ is then a quadratic refinement of the ``$g$-twisted intersection form'' $(\cdot, g \cdot)\pmod 2$:
\[
q_g(x+y) = q_g(x) + q_g(y) + (x, g y) \pmod 2.
\]
\end{lemma}
\begin{proof}
Since $G_v = \pair{1}$ for $v \in \{x,y,x+y\}$, the corresponding subgroup $\tilde{S_{G_v,g}} = \pair{g}$ has order $2$, and so $q_g(v)$ counts pairs of intersection points ($\pair{g}$-orbits) in the intersection $\gamma \cap g \gamma$ for suitable representatives $\gamma$ of $v$.

Represent $x,y$ by multicurves $\gamma, \delta \subset X$ in accordance with \Cref{standingconventions}. Consider the configuration $(\gamma \cup \delta) \cup (g \gamma \cup g \delta) \subset X$. Then $q_g(x+y)$ is the mod-$2$ count of $g$-orbits in the intersection of the first set with the second. These are of four types: 
\begin{enumerate}
    \item $\gamma \cap g \gamma$
    \item $\delta \cap g \delta$
    \item $\gamma \cap g \delta$
    \item $\delta \cap g \gamma$.
\end{enumerate}
Intersections of type (1) contribute $q_g(x)$ to the total, and likewise those of type (2) contribute $q_g(y)$. If $p \in \gamma \cap g \delta$ is an intersection of type (3), then necessarily $g p \in \delta \cap g \gamma$ is an intersection of type (4). It follows that the number of $g$-orbits appearing in types (3) and (4) is in bijection with the number of intersections $\gamma \cap g \delta$, so that intersections of types (3) and (4) contribute $(x, g y)$ to $q_g(x+y)$ as claimed.
\end{proof}

\section{Necessary conditions: proof of \Cref{maintheorem:necessary}}\label{section:necessary}

Throughout this section, $v \in H_1(X; \Z)$ denotes a relatively geometric class, represented by a component $\tilde \gamma \subset f^{-1}(\gamma)$ for some simple closed curve $\gamma \subset Y$ that is nonseparating. \Cref{maintheorem:necessary} asserts that there are four conditions that a nonseparating relatively-geometric class must satisfy: {\em isotropy, parity, primitivity,} and {\em coherence}. The first two of these (isotropy, parity) will be seen to hold following the work of the previous two sections, and are discussed in \Cref{subsection:intersections}. Following this, we discuss primitivity in \Cref{subsection:primitivity}, and coherence in \Cref{subsection:coherence}. \Cref{maintheorem:necessary} will then follow by assembling \Cref{lemma:isotropic,lemma:parity,lemma:primitive,lemma:coherent}. 

\subsection{Isotropy and parity}\label{subsection:intersections}
\begin{lemma}\label{lemma:isotropic}
Under the standing assumptions of \Cref{section:necessary}, $v$ is {\em isotropic}:
\[
\pair{v,v} = 0.
\]
\end{lemma}
\begin{proof}
When $\gamma$ is simple, the formula in \Cref{corollary:selfint} that computes $\pair{v,v} = \pair{[\tilde \gamma], [\tilde \gamma]}$ is a sum over the empty set.
\end{proof}

\begin{lemma}\label{lemma:parity}
Under the standing assumptions of \Cref{section:necessary}, $v$ is {\em even}:
\[
q(v) = 0.
\]
\end{lemma}
\begin{proof}
When $v = [\tilde \gamma]$ is relatively-geometric, the local formula for $q(v)$ appearing in \Cref{lemma:qlocal} is again a sum over the empty set.
\end{proof}

\subsection{Primitivity}\label{subsection:primitivity}
An element of a torsion-free abelian group is commonly called ``primitive'' if it cannot be represented as a nontrivial multiple of some other element. For our purposes, we will require a substantial refinement of this notion, adapted to the setting of $\Z[G]$-modules with skew-Hermitian pairing. We will use the term {\em integrally primitive} for elements that cannot be represented as nontrivial multiples, and reserve the term {\em primitive} for elements as in \Cref{definition:pairing}.

\begin{definition}[Pairing ideal, primitive]\label{definition:pairing}
Given $v \in H_1(X; \Z)$, the {\em pairing ideal} $I_v \normal \Z[G]$ is the left ideal
\[
I_v = \{\pair{u,v}\mid u \in H_1(X; \Z)\}.
\]
An element $v \in H_1(X;\Z)$ with finite cyclic stabilizer $G_v \leqslant G$ is {\em primitive} if the pairing ideal has the form
\[
I_v =\Z[G] \zeta_v
\]
(recall $\zeta_v \in \Z[G]$ is the element defined in \eqref{zetadef}). 
\end{definition}

\begin{lemma}\label{lemma:primitive}
Under the standing assumptions of \Cref{section:necessary}, $v$ is {\em primitive}:
\[
I_v =\Z[G] \zeta_v.
\]
\end{lemma}
\begin{figure}[ht]
\labellist
\small
\pinlabel $Y^+$ at 320 20
\endlabellist
\includegraphics[scale=1]{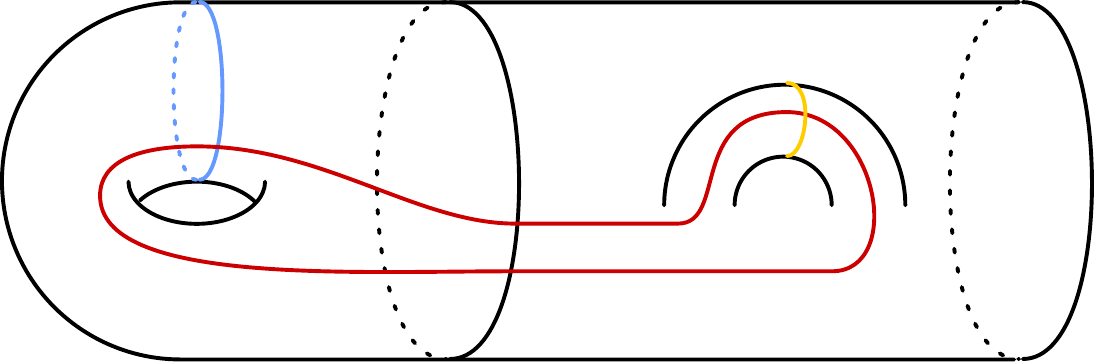}
\caption{Construction of $\delta$ proceeds in two steps. In the first step, since $\gamma$ is nonseparating, we can construct an arc $\alpha \subset X$ based at $\Delta_0$ that crosses $\gamma$ once. Such $\alpha$ determines an element $\phi(\alpha) \in G$. In the second step, we close up $\alpha$ to the simply-lifting simple closed curve $\delta$ by stabilizing by a basic $\phi(\alpha)^{-1}$-handle.}
\label{figure:primitive}
\end{figure}
\begin{proof}
\Cref{lemma:cyclicstab} asserts that any nonseparating, relatively geometric $v$ has finite cyclic stabilizer. To see the containment $I_v \leqslant \Z[G] \zeta_v$, let $\pair{u,v} \in I_v$ be arbitrary. From \eqref{def:relint},
\[
\pair{u,v} = \sum_{g \in G}(u,gv)g = \sum_{g \in G/G_v}\sum_{h \in G_v} (u,gv)gh = \left( \sum_{g \in G/G_v}(u,gv)g\right) \zeta_v.
\]

For the reverse containment, represent $v$ as a component $\tilde \gamma \subset f^{-1}(\gamma)$. Perform the stabilization $X \into X^+$ indicated in \Cref{figure:primitive}. As indicated therein, since $\gamma$ is nonseparating, it is possible to construct a curve $\delta$ with trivial stabilizer (i.e. such that $\phi(\delta_\bullet) = 1$) that crosses $\gamma$ once. By \Cref{lemma:localformula}, it follows that $\pair{[\tilde \delta],[\tilde \gamma]} = \zeta_v$, showing the reverse containment.
\end{proof}

\subsection{Coherence}\label{subsection:coherence}
\begin{definition}[Coherence]\label{definition:coherent}
Let $v \in H_1(X;\Z)$ be given, with finite cyclic stabilizer subgroup $G_v \leqslant G$. Define $X_v:= X/G_v$, and consider the associated intermediate cover $f_v: X \to X_v$ classified by the map $\phi_v: \pi_1(X_v) \to G_v$. Since $G_v$ is abelian, $\phi_v$ factors through $\phi_{v,*}: H_1(X_v; \Z) \to G_v$. Such $v$ is said to be {\em coherent} if $f_{v,*}(v) \in H_1(X_v;\Z)$ is $\abs{G_v}$-divisible and 
\[
G_v = \pair{\phi_{v,*}(f_{v,*}(v)/\abs{G_v})}.
\]
\end{definition}
As an example of a non-coherent element, consider the surface $Y = \Sigma_{2,1}$ with $H_1(Y;\Z)$ endowed with symplectic basis $x_1, y_1, x_2,y_2$, and let $\phi: \pi_1(Y) \to \Z/2\Z$ be given by $\phi(v) = (v,x_1) \pmod 2$. The handle with homology basis $x_2,y_2$ is then a simple stabilization of a mod-$2$ cover of the torus spanned by $x_1,y_1$; let $\tilde x_2, \tilde y_2$ denote the basis for this stabilization as in \Cref{proposition:simplehomol}. Letting $G = \Z/2\Z = \pair{t}$, the element $v = (1+t)\tilde x_2$ has stabilizer $G_v = \Z/2\Z$. The associated cover $X_v$ is just $Y$, and $f_*(v)/2 = x_2$, but $\phi(x_2) = (x_2, x_1) \pmod 2 = 0$ does not generate $\Z/2\Z$.

This example illustrates the role that coherence will play in the proof of \Cref{maintheorem:sufficient}: coherence ensures that a class $v$ have a {\em connected} $G_v$-invariant representative (here, $(1+t)\tilde x_2$ is represented as the full preimage of a representative of $x_2$, but this has {\em two} components). See \Cref{lemma:coherenceconnected}.

\begin{lemma}\label{lemma:coherent}
Under the standing assumptions of \Cref{section:necessary}, $v$ is {\em coherent}.
\end{lemma}
\begin{proof}
Suppose $v$ is represented by an elevation $\tilde \gamma$ of a nonseparating simple closed curve $\gamma \subset Y$. Covering space theory asserts that the restriction of $f_v$ to $\tilde \gamma \subset X$ is a covering map onto its image $f_v(\tilde \gamma)$ of degree $\abs{G_v}$. Thus $f_{v,*}([\tilde \gamma]) = \abs{G_v} [f_v(\tilde \gamma)]$ is a $\abs{G_v}$-divisible element of $H_1(X_v;\Z)$ as required. Again by covering space theory, the order of $\phi_{v,*}(f_v(\tilde \gamma)) \in G_v$ is equal to this degree, and so $G_v$ is generated by $\phi_{v,*}(f_v(\tilde \gamma)) = \phi_{v,*}(f_{v,*}(v)/\abs{G_v})$, completing the proof.
\end{proof}

\section{Purely-unital vectors}\label{section:PU}
In these last three sections, we carry out the proof of \Cref{maintheorem:sufficient}. Here, we investigate a special class of elements we call ``purely-unital vectors'', and give a direct, constructive proof that the conditions of \Cref{maintheorem:necessary} suffice to realize a purely-unital vector geometrically. In the last two sections, we will establish the general case via the following strategy: beginning with an arbitrary representative $\delta_0$ for a class $v$, we {\em resolve} self-intersections of $f(\delta_0)$ by routing the crossing points through a new basic $1$-handle. This has the effect of changing the homology class, but when the total change can be realized geometrically, it is possible to account for this change via a Dehn twist, leading to a realization of the original class $v$, and the alterations produced by the resolution process is precisely what the notion of purely-unital vector captures. Accordingly, we describe this ``resolution process'' in \Cref{section:resolution}, and subsequently carry out the proof of \Cref{maintheorem:sufficient} in \Cref{section:proof}.\\

In preparation for the definition of a purely-unital vector, we recall that a {\em hyperbolic module} over a ring $R$ is a free module $H \cong R^{2n}$ equipped with a skew-Hermitian form $\pair{\cdot,\cdot}$ for which there is a basis $\{x_1, \dots, x_n, y_1, \dots, y_n\}$ such that $\pair{x_i,y_i} = 1$ and all other pairings of basis elements are $0$. Such a basis is called a {\em hyperbolic basis}. In this language, we can express \Cref{proposition:simplehomol} as saying that a sequence of simple stabilizations affects $H_1(X;\Z)$ by adding a direct summand with a hyperbolic $\Z[G]$-module.

\begin{definition}[Purely-unital vector]
Let $f^+: X^+ \to Y^+$ be a simple stabilization of $f: X \to Y$; let $x_1, y_1, \dots, x_k, y_k$ form an associated hyperbolic basis for $H_1(X^+;\Z)^{stab}$. An element $v \in H_1(X^+; \Z)^{stab}$ is {\em purely unital} if it admits an expression of the form
\[
v = \sum_{i = 1}^k (g_i x_i + h_i y_i)
\]
for elements $g_1, h_1, \dots, g_k, h_k \in \pm G \subset \Z[G]$, such that $g_i \ne \pm h_i$ for all indices $i$. 
\end{definition}
\begin{remark}
\label{remark:trivstab}
Note that a purely-unital vector $v$ necessarily has trivial stabilizer: $G_v =\pair{1}$. It follows that the coherence condition is vacuously satisfied for any purely-unital vector. Moreover, every purely-unital vector is also easily seen to be primitive.
\end{remark}

\begin{figure}[ht]
\labellist
\small
\pinlabel $\blue{g_{2i-1}}$ [tr] at 62.36 141.72
\pinlabel $\red{h_{2i-1}}$ [t] at 107.71 141.72
\pinlabel $\textcolor{mygreen}{g_{2i}}$ [tr] at 154.55 141.72
\pinlabel $\textcolor{myorange}{h_{2i}}$ [tl] at 209.74 141.72
\endlabellist
\includegraphics[scale=1]{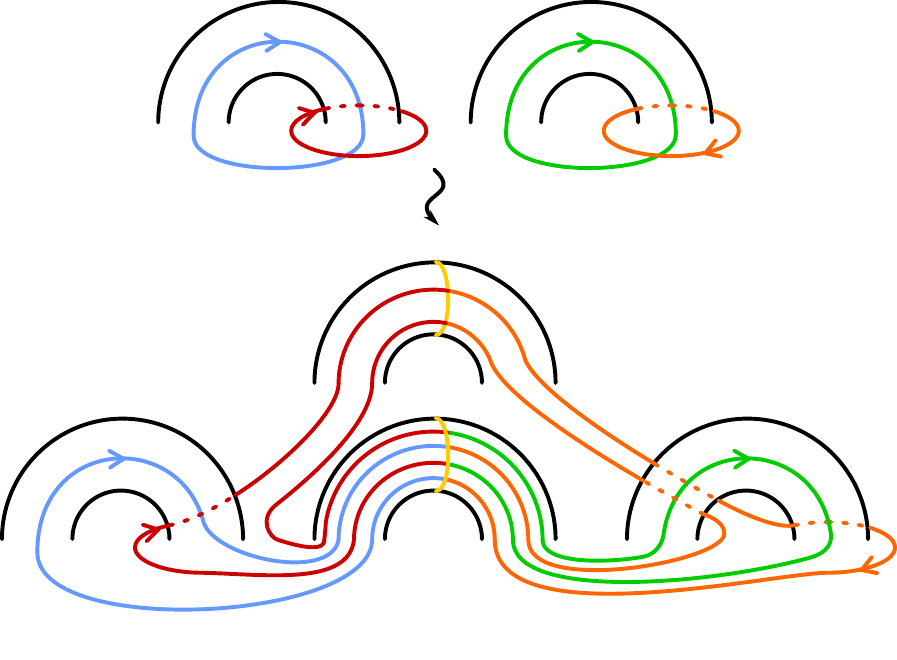}
\caption {Connecting the four components in a pair of handles into a single simple closed curve, as seen downstairs on $Y^+$. The colors of the strands indicate the sheets of $X^+$ containing the indicated strands, and are labeled in the top portion of the figure.
The two middle handles on the bottom are added in the stabilization process. The top is a basic $h_{2i}^{-1}h_{2i-1}$ handle, and the bottom is a basic $h_{2i}^{-1}g_{2i-1}$-handle.}
\label{figure:unitalresolution}
\end{figure}
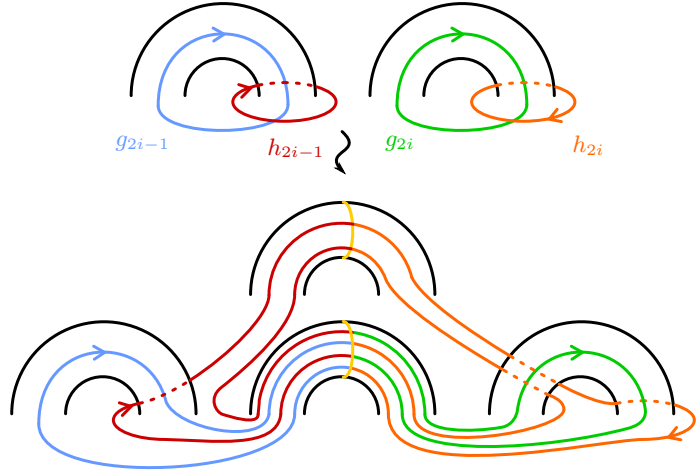

\begin{lemma}\label{lemma:purestable}
Let $f^+: X^+ \to Y^+$ be a simple stabilization of $f: X \to Y$, and let $v \in H_1(X^+; \Z)^{stab}$ be purely unital and satisfy the necessary conditions for relative geometricity of \Cref{maintheorem:necessary}. Then $v$ is relatively geometric.
\end{lemma}
\begin{proof}
As noted in \Cref{remark:trivstab}, purely-unital vectors are coherent and primitive. Thus the relevant hypotheses are that $v$ is isotropic and even. Computing,
\[
\pair{v,v} = \pair{\sum_{i = 1}^k (g_i x_i + h_i y_i),\sum_{i = 1}^k (g_i x_i + h_i y_i)} =  \sum_{i= 1}^k (g_i h_i^{-1} - h_i g_i^{-1})
\]
and, applying \Cref{lemma:qprops},
\[
q(v) = q\left(\sum_{i = 1}^k (g_i x_i + h_i y_i)\right) = \sum_{i = 1}^k \sum_{g \in G_2^*} (g_i x_i, g h_i y_i) g = \sum_{g \in G_2^*} N(g) g,
\]
with $N(g) \in \Z/2\Z$ the mod-$2$ count of the number of indices $i$ for which $g = g_i h_i^{-1}$. 

Thus if $v$ satisfies the necessary conditions of \Cref{maintheorem:necessary}, the following conditions are satisfied by the elements $g_i, h_i \in \pm G \cup \{0\}$:
\begin{align}
    \sum_{i= 1}^k (g_i h_i^{-1} - h_i g_i^{-1}) = 0 \label{condition:non2} \\ 
    \mbox{Each $g \in G_2^*$ appears as an {\em even} number of $g_i h_i^{-1}$.}\label{condition:2}
\end{align}
Note that \eqref{condition:non2} can be re-formulated as follows:
\begin{align} \label{condition:2prime}
    \mbox{Each $g \in G \setminus G_2$ appears as $g_i h_i^{-1}$ for exactly as many indices as does $g^{-1}$.}
\end{align}
Note then that \eqref{condition:2}, \eqref{condition:2prime} together imply that, after re-labeling the indices, the handles can be paired: {\em $k = 2k'$ and for all $1 \le i \le k'$, the equalities}
\begin{equation}\label{equation:pairs}
g_{2i-1} h_{2i-1}^{-1} = (g_{2i} h_{2i}^{-1})^{-1} 
\end{equation}
{\em hold.}

To realize such $v$ relatively-geometrically, we begin by realizing $v$ as a cycle in $X^+$ with non-simple projection onto $Y^+$. Specifically, represent $v$ by the cycle
\[
\gamma_0 = \sum_{i = 1}^k g_i \xi_i + h_i \eta_i,
\]
where $\xi_i, \eta_i$ generate the homology of the $i^{th}$ stable handle. To complete the argument we will show how to replace $\gamma_0$ with a homologous cycle $\gamma$ that is connected and for which $f^+(\gamma)$ is simple.

We resolve the crossings of $f^+(\gamma_0)$ first. The image $f^+(\gamma_0)$ has exactly $k$ transverse double points, one for each of the $k$ handles. By \eqref{equation:pairs}, these come in $k'$ pairs. \Cref{figure:unitalresolution} shows how, after stabilizing, the four-component cycle $g_{2i-1} \xi_{2i-1} + h_{2i-1}\eta_{2i-1} + g_{2i} \xi_{2i} + h_{2i}\eta_{2i}$ can be replaced with the homologous simple closed curve $\omega_i$ such that $f^+(\omega_i)$ is simple. 

We define the multicurve $\gamma_1$ as
\[
\gamma_1 = \sum_{i = 1}^{k'} \omega_{i}.
\]
To construct $\gamma$ from $\gamma_1$, we connect the $k'$ components of $\gamma_1$ together. For $i = 1, \dots, k'-1$, let $\alpha_i \subset Y^+$ be an arc beginning at $f^+(\omega_i)$, ending at $f^+(\omega_{i+1})$, and otherwise disjoint from all $\{\omega_j\}$ and other arcs $\{\alpha_j\}$. Performing additional stabilizations if necessary, we can construct a set of such arcs such that the lift of $\alpha_i$ beginning at $\omega_i$ ends at $\omega_{i+1}$. The iterated connect-sum of the curves $\{\omega_j\}$ along these distinguished lifts is then a connected simple closed curve $\gamma$ that represents the homology class $v$ and such that the projection $f^+(\gamma)$ is simple.
\end{proof}

\section{The resolution process}\label{section:resolution}
The resolution process takes as input a multicurve $\delta \subset X$ for which $f(\delta)$ has a self-crossing, and returns a simple stabilization $X^+$ of $X$ and a new multicurve $\delta' \subset X^+$ with one fewer self-crossing, but homology class altered by some purely-unital vector. We begin by giving the formal definition (\Cref{definition:resolution}), and then analyze the effect on homology in \Cref{lemma:resolution1,lemma:resolution}.

\begin{definition}[Resolution process]\label{definition:resolution}
Let $\delta\subset X$ be a multicurve with finite cyclic stabilizer group $G_\delta$. Suppose that $f(\delta)$ has a transverse self-crossing at $p \in Y$. Let $\alpha \subset Y$ be an arc connecting $p$ to $\Delta_0$. The {\em resolution of $\delta$ along $\alpha$} is the multicurve $\delta'$ on the surface $X^+$ obtained from $X$ by adding a basic $1$-handle $H \subset Y^+$ constructed as depicted in \Cref{figure:resolution}. 
\end{definition}

\begin{remark}
\label{remark:components}
Note that applying the resolution process to some multicurve $\delta$ does not change the number of components.
\end{remark}

\begin{figure}[ht]
\labellist
\small
\pinlabel $\alpha$ at 110 35
\endlabellist
\includegraphics[scale=0.575]{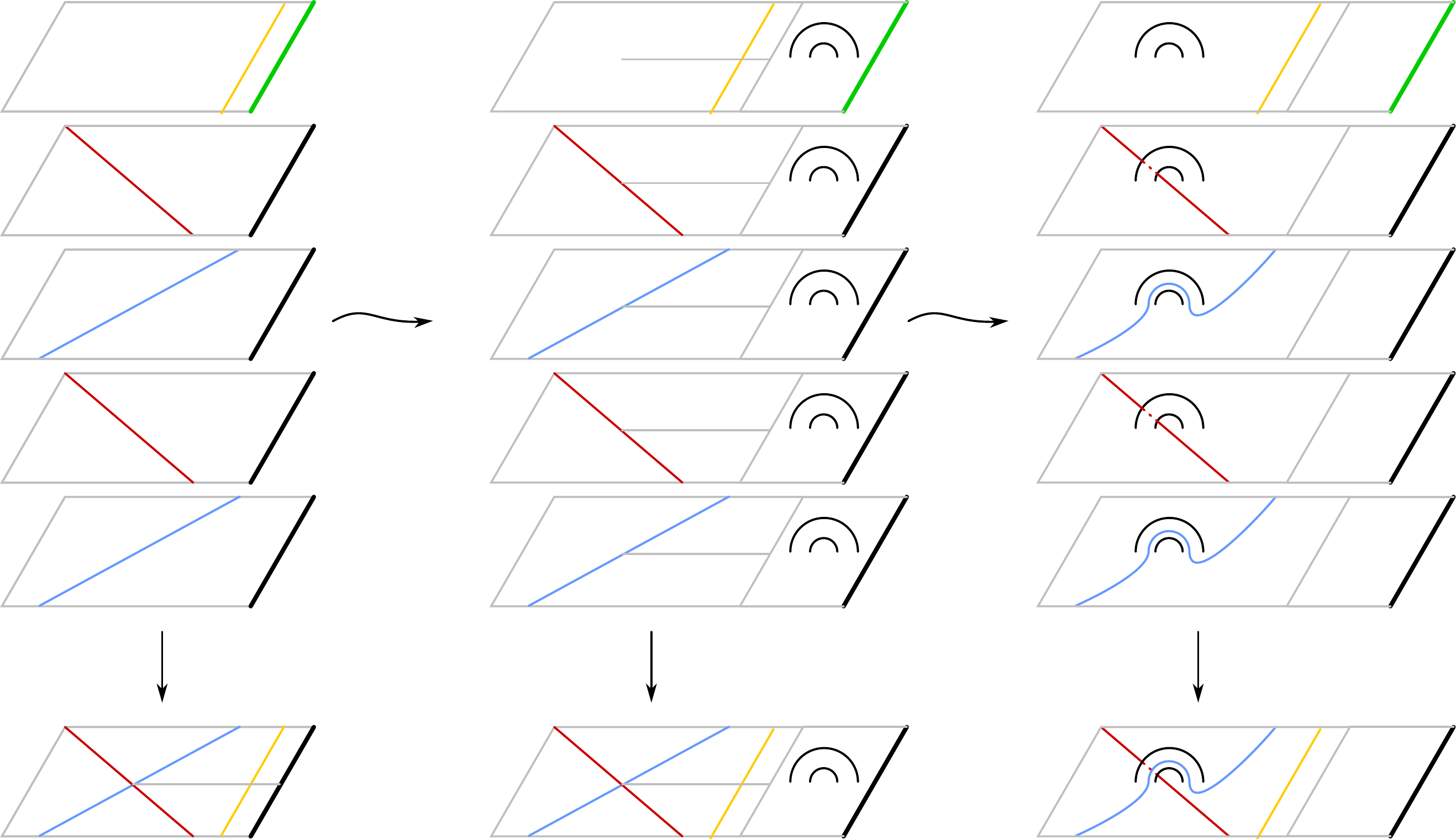}
\caption {The resolution process. Push the handle $H$ along $\alpha$ into the interior of $Y$. Whenever $\alpha$ crosses over a strand of $f(\delta)$, change the isotopy class of $\delta$ by $G_\delta$-equivariantly moving each local branch across the corresponding boundary component of $f^{-1}(H)$. When $H$ reaches $p$, resolve the crossing by routing one $G_\delta$-orbit of local branches around the core of the corresponding handle above $H$, and the other orbit around the co-core. (Alternatively, one can imagine carrying out this process by {\em pushing} the point $p$ out to $H$ while dragging transverse crossings along in front.)}
\label{figure:resolution}
\end{figure}

How does the resolution process alter the homology class $[\delta]$? The answer turns out to depend on the choice of path $\alpha$, and is presented in \Cref{lemma:resolution1}. In \Cref{lemma:resolution}, we present a complementary result that asserts that it is possible to {\em choose} an arc so that the change in homology under the resolution process can be controlled. In preparation, we record the following lemma, which follows from the basic principles of covering space theory and the discussion of the local sheet index ca. \Cref{definition:localindex}.
\begin{lemma}\label{lemma:arccoset}
Let $\delta, p, \alpha$ be as above. Then $\alpha$ provides an identification of the local branches of $\delta$ above $p$ with right cosets $G_\delta g_1$ and $G_\delta g_2$ as follows: lift $\alpha$ with basepoint the chosen local branch; $\alpha$ then ends at some component of $f^{-1}(\Delta_0)$, and, having chosen a distinguished component $\tilde \Delta_0 \subset f^{-1}(\Delta_0)$, these components are identified with $G$. 
\end{lemma}

\begin{lemma}\label{lemma:resolution1}
Let $\delta\subset X$ be given with finite cyclic stabilizer group $G_\delta \leqslant G$. Suppose that $\delta$ has a transverse self-crossing above $p \in Y$ of local sheet index $i(\delta, \delta, p) = G_\delta g G_\delta$. Let $\alpha \subset Y$ be an arc connecting $p$ to $\Delta_0$, and suppose that, via \Cref{lemma:arccoset}, $\alpha$ identifies the local branches with the cosets $G_\delta g_1$ and $G_\delta g_2$. Then the resolution $\delta'$ of $\delta$ along $\alpha$ has homology class
\begin{equation}\label{equation:firstdeltaprime}
[\delta'] = [\delta] + \zeta_\delta (g_1 x + g_2 y).
\end{equation}
Here $x, y$ generate the homology of the basic $1$-handle, as in \Cref{proposition:simplehomol}.
\end{lemma}
\begin{proof}
There are two types of modifications made in transforming $\delta$ into $\delta'$: passing a strand across a boundary component of $f^{-1}(H)$, and resolving the crossing. We analyze each in turn.

We claim that when a component of $f^{-1}(H)$ crosses over a strand of $\delta$, the homology class is unchanged. This follows because the effect on homology is to add a copy of the class of the boundary component, but the homology class of a boundary is zero.

When the crossing is resolved, \Cref{figure:resolution} shows that the effect on homology is to add a class of the form $\zeta_\delta (f_1 x + f_2 y)$ for some $f_1, f_2 \in G$. The values $f_1, f_2$ are determined by the conventions by which we put coordinates on the stable handle, as in \Cref{proposition:simplehomol}. The homology class $x$ is by definition the elevation based on the distinguished component $\tilde \Delta_0$. Under $\alpha$, the local branches are located in the sheets corresponding to cosets $G_\delta g_1$ and $G_\delta g_2$, and hence the added classes are $\zeta_\delta g_1 x$ and $\zeta_\delta g_2 y$ as claimed.
\end{proof}

\begin{lemma}\label{lemma:resolution}
Let $\delta\subset X$ be given with finite cyclic stabilizer group $G_\delta\leqslant G$. Suppose that $\delta$ has a transverse self-crossing above $p \in Y$ of local sheet index $i(\delta, \delta, p) = G_\delta g G_\delta$. Then for any $g' \in G_\delta g G_\delta$, there is an arc $\alpha$ such that the resolution $\delta'$ of $\delta$ along $\alpha$ has homology class
\begin{equation}\label{equation:deltaprime}
[\delta'] = [\delta] + \zeta_\delta (x + g' y).
\end{equation}
\end{lemma}

\begin{proof}
Let $\alpha' \subset Y$ be an arbitrary arc connecting $p$ to $\Delta_0 \subset \partial Y$. Via \Cref{lemma:arccoset}, this identifies the local branches of $\delta$ above $p$ with cosets $G_\delta g_1$ and $G_\delta g_2$. If $\beta \in \pi_1(Y)$ is a loop based at the endpoint of $\alpha'$ on $\Delta_0$, the concatenated path $\alpha' \beta$ identifies the local branches with the cosets $G_\delta g_1 \phi(\beta)$ and $G_\delta g_2 \phi(\beta)$.

With the notation already established, the local sheet index is $G_\delta g_1 g_2^{-1} G_\delta$. Thus, there exist $h_1, h_2 \in G_\delta$ such that
\[
g' = h_1 g_1 g_2^{-1} h_2.
\]
Choose coset representatives $g_1 \in G_\delta g_1$ and $g_2 \in G_\delta g_2$, and then choose $\beta$ so that $\phi(\beta) = g_2^{-1} h_2$. Using the concatenation $\alpha:= \alpha' \beta$ to identify local branches with right cosets of $G_\delta$ then gives $G_\delta g'$ and $G_\delta$. When the crossing is resolved using $\alpha$, \Cref{lemma:resolution1} shows that the added homology class is of the form $\zeta_\delta(x + g' y)$ as claimed.
\end{proof}

\section{Proof of \Cref{maintheorem:sufficient}}\label{section:proof}

We come to the final stage of the proof of \Cref{maintheorem:sufficient}. The outline is as follows. We show in \Cref{lemma:coherenceconnected} that when $v$ is coherent, it admits a representative as a {\em connected} $G_v$-invariant simple closed curve; the objective is then to replace this representative with one which moreover has simple projection onto (a stabilization of) $Y$. For this, we carry out the resolution process of \Cref{section:resolution}, producing a {\em relatively geometric} class in the homology class $v + \zeta_v w$, where $w$ is a purely-unital vector. As explored in \Cref{lemma:resolution1,lemma:resolution}, there is some freedom in constructing $w$; we show in \Cref{lemma:choosereps} that the parity and isotropy assumptions on $v$ can be leveraged to construct $w$ so as to be relatively geometric (following the work of \Cref{lemma:purestable}). Finally, we use the relative geometricity of $w$ and of $v + \zeta_v w$ (and the primitivity hypotheses on $v$) in a Dehn twist construction which sends the relatively geometric class $v + \zeta_v w$ via a $G$-equivariant diffeomorphism onto the class $v$, as desired.

\subsection{Preparatory lemmas}
Here, we collect the preliminary results (\Cref{lemma:coherenceconnected,lemma:choosereps,lemma:dehntwist,lemma:betteru}) alluded to in the above outline.

\begin{lemma}
\label{lemma:coherenceconnected}
Let $v \in H_1(X;\Z)$ have cyclic stabilizer subgroup $G_v$ of finite order $d$, and suppose that $v$ is coherent. Then $v$ admits a representative $\gamma$ as a $G_v$-invariant simple closed curve.
\end{lemma}
\begin{proof}
We first recall the terminology of \Cref{definition:coherent}: we consider the cyclic covering 
\[
f_v: X \to X_v
\]
and the classifying map $\phi_{v,*}: H_1(X_v;\Z) \to G_v$. 

By the theory of the transfer map, every class in $H_1(X;\Z)^{G_v}$ admits a representative of the form
\[
v = [f_v^{-1}(\gamma)]
\]
for some multicurve $\gamma \subset Y$, and if $v$ is integrally primitive (i.e. if $v = k v'$ for some $v' \in H_1(X;\Z)$, then $k = \pm 1$), then $\gamma$ can be taken to be a simple closed curve. The number of components of $f_v^{-1}(\gamma)$ is then equal to the index 
\[
i_v = [G_v: \pair{\phi_{v,*}([\gamma])}].
\]
Again by the theory of the transfer map, $f_{v,*}(v) = d [\gamma]$. The hypothesis that $v$ is coherent then implies that $i_v = 1$, so that $f_v^{-1}(\gamma)$ is a connected $G_v$-invariant simple closed curve representing $v$, as was to be shown.
\end{proof}

The following lemma will allow us to leverage the assumption that $v$ is isotropic and even in order to construct a purely-unital vector $w$ that is relatively geometric.
\begin{lemma}
\label{lemma:choosereps}
Let $G_vg_1 G_v, \dots, G_v g_k G_v$ be double cosets such that the associated sum of local crossing factors vanishes:
\[
\sum_{i = 1}^k \zeta_v(g_i - g_i^{-1}) \zeta_v = 0.
\]
Then there are representatives $g_i' \in G_v g_i G_v$ for which 
\[
\sum_{i = 1}^k (g_i' - g_i'^{-1}) = 0.
\]
If, moreover, for each $g \in G_2^*$, the number of double cosets $G_vg_iG_v$ that equal $G_v g G_v$ is even, then the representatives $g_i'$ can be chosen such that each $g \in G_2^*$ appears as $g_i'$ for an {\em even} number of indices $1 \le i \le k$. 
\end{lemma}
\begin{proof}
Choose arbitrary representatives $g_i \in G_v g_i G_v$, and write
\[
\sum_{i = 1}^k (g_i - g_i') = \sum_{g \in G} c_g g,
\]
where, by construction
\begin{equation}\label{equation:cg}
c_g = \mbox{($\#$ times $g = g_i$)} - \mbox{($\#$ times $g = g_i^{-1}$)}.
\end{equation}
One computes that then 
\[
\zeta_v\left( \sum_{i = 1}^k (g_i- g_i^{-1})\right) \zeta_v = \zeta_v\left( \sum_{g \in G} c_g g\right) \zeta_v= \sum_{g \in G} \left(\sum_{(h_1, h_2) \in G_v \times G_v} c_{h_1 g h_2}  \right) g.
\]
By hypothesis, it follows that each coefficient $\sum_{G_v \times G_v} c_{h_1 g h_2} = 0$, but this sum is just a nonzero multiple of the sum over a fixed double coset, and so for all double cosets $G_v g G_v$, 
\[
\sum_{f \in G_v g G_v} c_f = 0.
\]

If all coefficients $c_f, f \in G$ are zero, then take $g_i' = g_i$ and the claim holds. Otherwise, choose some $f \in G$ with $c_f \ne 0$. Since $\sum_{f' \in G_v f G_v} c_{f'} = 0$, we can choose $c_f >0$, and there moreover exists $f' \in G_v f G_v$ with $c_{f'} < 0$. By the definition of $c_f$ given in \eqref{equation:cg}, it follows that $f = g_i$ for strictly more indices $i$ than $f = g_i^{-1}$, and likewise that $f' = g_i$ for strictly fewer indices than $f' = g_i^{-1}$. Replace some $g_i = f$ by $g_i' = f'$; then the coefficients $c_f$ and $c_{f'}$ both decrease in absolute value. Repeat this process until all $c_f = 0$: the resulting set of representatives $\{g_i'\}$ satisfy
\[
\sum_{i =1}^k(g_i' - g_i'^{-1}) = 0.
\]

Now suppose additionally that for each $g \in G_2^*$, the number of double cosets $G_v g_i G_v$ that equal $G_v g G_v$ is even. We first note that this can be expressed more simply as the number of {\em appearances} of some element of $G_v g G_v$ among the ordered list $(g_i')$ of $g_i'$. Suppose that for some $g \in G_2^*$, the chosen list $(g_i')$ has $g_i' = g$ for an {\em odd} number of indices $i$. Since the total number of appearances of the coset $G_v g G_v$ in $(g_i')$ is even by hypothesis, it follows that there is some other element $g' \in G_v g G_v$ that appears an odd number of times in $(g_i')$.  

We claim that there is moreover some such $g'$ contained in $G_2^*$; we postpone the proof to the paragraph below. By exchanging one $g_i' = g$ for $g_i'' = g'$, the number of appearances of $g$ and $g'$ can be made simultaneously even. The sum $\sum_{i = 1}^k (g_i' - g_i'^{-1})$ is unaltered by exchanging some $g$ for $g'$, as both factors $g-g^{-1} = g' - g'^{-1} = 0$ vanish since $g,g' \in G_2^*$. One has therefore reduced the number of $g \in G_2^*$ for which $g$ appears in $(g_i')$ an odd number of times; repeat this step as many times as necessary.

To establish the claim, we suppose to the contrary that there is a {\em unique} element $g \in G_v g G_v \cap G_2^*$ that appears in $(g_i')$ an odd number of times. Since $\sum_{i = 1}^k (g_i' - g_i'^{-1}) = 0$, the remaining elements of $G_v g G_v$ that appear in $(g_i')$ come in {\em pairs} $g', g'^{-1} \ne g'$, each appearing the same number of times. Thus the total number of indices corresponding to elements of this type is {\em even}, and so the {\em total} number of appearances of some $g' \in G_v g G_v$ in $(g_i')$ must be odd, contrary to hypothesis.
\end{proof}

The following lemma is well-known; see, e.g. \cite[Section 3]{looijenga}. 
\begin{lemma}
\label{lemma:dehntwist}
Let $\gamma \subset Y$ be a simple closed curve, and let $\tilde \gamma$ be a chosen elevation. Denote the stabilizer of $v:= [\tilde \gamma]$ by $G_v$. Then the Dehn twist power $T_\gamma^{\abs{G_v}}$ lifts to a $G$-equivariant mapping class $\tilde{T_\gamma}$ on $X$, and the action on $H_1(X;\Z)$ is given by the formula
\[
\tilde{T_\gamma}(x) = x + \frac{\pair{x,v}}{\abs{G_v}}v.
\]
\end{lemma}

The following lemma will allow us to assume that a ``certificate of primitivity'' for $v$ (i.e. a class $u$ with $\pair{u,v} = \zeta_v$) is relatively geometric.
\begin{lemma}
\label{lemma:betteru}
Let $u_1,v \in H_1(X;\Z)$ be given, and suppose that $\pair{u_1,v} = \zeta_v$. Then there is a stabilization $X^+$ and a homology class $u_2$ supported on $X^+ \setminus X$ such $u:= u_1 + u_2$ is relatively geometric, represented by $\eta \subset X^+$ with $f^+(\eta) \subset Y^+$ simple, and such that $\pair{u,v} = \zeta_v$.
\end{lemma}

\begin{proof}
First, observe that the hypothesis $\pair{u_1,v} = \zeta_v$ implies that $u_1$ is integrally primitive and that $G_{u_1} = \pair{1}$. In particular, $u_1$ can be represented by a simple closed curve $\eta_1 \subset X$ in general position. The self-crossings of $f^+(\eta_1)$ can be resolved by the resolution process (\Cref{definition:resolution}), leading to a simple closed curve $\eta \subset X^+$ supported on some stabilization $X^+$ of $X$, such that $f^+(\eta)\subset Y^+$ is simple. \Cref{lemma:resolution1} shows that on the level of homology, 
\[
[\eta] = [\eta_1] + \sum (g_{i,1} x_i + g_{i,2} y_i)
\]
with $x_i, y_i$ forming a basis for the $i^{th}$ handle added in the resolution process. Such classes are indeed supported on $X^+ \setminus X$ as claimed, and hence $\pair{u,v} = \pair{u_1, v} = \zeta_v$ as required.
\end{proof}

\subsection{Proof of \Cref{maintheorem:sufficient}}
By \Cref{lemma:coherenceconnected}, since $v$ is coherent, it admits a representative of the form $v = [\delta_0]$ for some simple closed curve $\delta_0$ that is invariant under the action of $G_v$; following \Cref{standingconventions}, we assume that $\delta_0$ is in general position. The projection $f(\delta_0)$ to $Y$ then has a finite number of self-crossings $P = \{p_1, \dots, p_k\} \subset Y$. 

Let $G_vg_iG_v$ be the local sheet index at $p_i$. By hypothesis, $v$ is isotropic, and so by \Cref{lemma:localformula},
\[
\pair{v,v} = \sum_{i = 1}^k \zeta_v (g_i - g_i^{-1}) \zeta_v = 0.
\]
Likewise, $v$ is assumed to be even, and so by \Cref{lemma:qlocal}, for each $g \in G_2^*$, the number of appearances of $G_v g G_v$ among the local sheet indices $\{G_v g_i G_v\}$ is even.

By \Cref{lemma:choosereps}, there exist representatives $g_i \in G_v g_i G_v$ for which
\begin{equation}\label{equation:gsum0}
\sum_{i = 1}^k (g_i - g_i^{-1}) = 0,
\end{equation}
and such that each $g \in G_2^*$ appears as an even number of the elements $g_i$.

For each self-crossing $p_i \in P$, we perform the resolution procedure of \Cref{lemma:resolution} using an arc $\alpha_i$ so that the class added is of the form $\zeta_v(x_i + g_i y_i)$ (here $x_i, y_i$ generate the homology of the $i^{th}$ stabilization). Since the resolution process does not alter the number of components (\Cref{remark:components}), the result is a simple closed curve $\delta_1$ with $f^+(\delta_1)$ possessing no self-crossings. By \Cref{lemma:resolution}, $[\delta_1] \in H_1(X^+;\Z)$ is given by
\[
[\delta_1] = [\delta_0] + \sum_{i = 1}^k \zeta_v(x_i + g_i y_i) = v + \sum_{i = 1}^k \zeta_v(x_i + g_i y_i).
\]

Define
\[
w = \sum_{i = 1}^k (x_i + g_i y_i),
\]

so that $[\delta_1] = v + \zeta_v w$. Then $w$ is a purely-unital vector by construction. We claim that $w$ is relatively geometric, i.e. satisfies the conditions of Lemma \ref{lemma:purestable}. Recall from \Cref{remark:trivstab} that every purely-unital vector is necessarily primitive and coherent. We claim that $w$ is isotropic:
\begin{align*}
\pair{w,w}  &= \pair{\sum_{i = 1}^k (x_i + g_i y_i), \sum_{j = 1}^k (x_j + g_j y_j)}\\ 
            &= \sum_{i,j} \pair{x_i + g_iy_i, x_j + g_j y_j}\\
            &= \sum_{i = 1}^k \pair{x_i + g_i y_i, x_i + g_i y_i}\\
            &= \sum_{i = 1}^k (g_i^{-1} - g_i)\\
            &= 0,
\end{align*}
the latter holding by \eqref{equation:gsum0}. Lastly we claim that $w$ is even, so that $q_g(w) = 0$ for all $g \in G_2^*$. Using the properties of $q_g$ developed in \Cref{lemma:qprops} and the fact that each $g \in G_2^*$ appears as $g = g_i$ an even number of times,
\begin{align*}
q_g(w)      &= q_g\left(\sum_{i = 1}^k (x_i + g_i y_i)\right)\\ 
            &= \sum_{i =1}^k q_g(x_i+g_i y_i)\\
            &= \sum_{i = 1}^k (x_i,g g_iy_i)\\
            &= (\# \mbox{ indices $i$ for which $g_i = g$}) \pmod 2\\
            &= 0.
\end{align*}
Thus by \Cref{lemma:purestable}, $w$ is relatively geometric: $w = [\omega]$ for a simple closed curve $\omega \subset X^+$ such that $f^+(\omega) \subset Y^+$ is simple. 

We next invoke the hypothesis that $v$ is primitive. To that end, let $u_1 \in H_1(X^+,\Z)$ be such that $\pair{u_1,v} = \zeta_v$. Represent $u_1$ by a simple closed curve $\eta \subset X^+$. By \Cref{lemma:betteru}, $u_1$ can be replaced by a {\em relatively-geometric} class $u = [\eta]$ on some further stabilization $X^+$ of $X$, such that $\pair{u,v} = \zeta_v$. It is possible to construct $\eta$ so that $f^+(\eta)$ and $f^+(\omega)$ are disjoint: each is constructed via a stabilization process on disjoint subsurfaces, and neither class separates, so it is possible to resolve crossings of one without passing through the other on the way to the boundary. Possibly after taking one last stabilization, one can construct an arc connecting $f^+(\eta)$ to $f^+(\omega)$ so that the connect-sum of $f^+(\eta), f^+(\omega)$ along this arc has an elevation $\epsilon$ in the homology class $[\eta] + [\omega] = u + w$.

To summarize, we have shown that the classes $v+ \zeta_v w$, $u$ and $u+w$ are relatively geometric, represented respectively by curves $\delta_1, \eta, \epsilon$. The former has stabilizer subgroup $G_v$, and the latter two have trivial stabilizer. By \Cref{lemma:dehntwist}, the Dehn twists $T_\epsilon$ and $T_\eta$ lift to diffeomorphisms $\tilde{T_\epsilon}, \tilde{T_\eta}$ on $X^+$, and the action on homology of $\tilde{T_\epsilon} \tilde{T_\eta}^{-1}$ is given as
\begin{align*}
\tilde{T_\epsilon} \tilde{T_\eta}^{-1}(v+ \zeta_v w) &= \tilde{T_\epsilon} \left(v+ \zeta_v w - \pair{v+\zeta_v w,u}u \right)\\
    &= \tilde{T_\epsilon} \left(v+ \zeta_v (u + w) \right)\\
    &= v+ \zeta_v (u + w) + \pair{v + \zeta_v (u+w), u+w}(u+w)\\
    &= v.
\end{align*}
Thus $v$ admits a representative $\tilde{T_\epsilon} \tilde{T_\eta}^{-1}(\delta_1)$ as a simple closed curve with simple projection: $v$ is relatively geometric as claimed. \qed

	\bibliographystyle{alpha}
	\bibliography{bibliography}

\begin{thebibliography}{GLLM15}

\bibitem[FH16]{FH}
B.~Farb and S.~Hensel.
\newblock Finite covers of graphs, their primitive homology, and representation
  theory.
\newblock {\em New York J. Math.}, 22:1365--1391, 2016.

\bibitem[FH17]{FH2}
B.~Farb and S.~Hensel.
\newblock Moving homology classes in finite covers of graphs.
\newblock {\em Israel J. Math.}, 220(2):605--615, 2017.

\bibitem[Fla21]{flamm}
X.~Flamm.
\newblock Subrepresentations in the homology of finite covers of graphs.
\newblock {P}reprint, https://arxiv.org/abs/2107.07428, 2021.

\bibitem[GLLM15]{GLLM}
F.~Grunewald, M.~Larsen, A.~Lubotzky, and J.~Malestein.
\newblock Arithmetic quotients of the mapping class group.
\newblock {\em Geom. Funct. Anal.}, 25(5):1493--1542, 2015.

\bibitem[KS16]{KS}
T.~Koberda and R.~Santharoubane.
\newblock Quotients of surface groups and homology of finite covers via quantum
  representations.
\newblock {\em Invent. Math.}, 206(2):269--292, 2016.

\bibitem[L\"02]{luck}
W.~L\"{u}ck.
\newblock A basic introduction to surgery theory.
\newblock In {\em Topology of high-dimensional manifolds, {N}o. 1, 2
  ({T}rieste, 2001)}, volume~9 of {\em ICTP Lect. Notes}, pages 1--224. Abdus
  Salam Int. Cent. Theoret. Phys., Trieste, 2002.

\bibitem[Loo97]{looijenga}
E.~Looijenga.
\newblock Prym representations of mapping class groups.
\newblock {\em Geom. Dedicata}, 64(1):69--83, 1997.

\bibitem[MP19]{MP}
J.~Malestein and A.~Putman.
\newblock Simple closed curves, finite covers of surfaces, and power subgroups
  of {${\rm Out}(F_n)$}.
\newblock {\em Duke Math. J.}, 168(14):2701--2726, 2019.

\bibitem[Put11]{putman}
A.~Putman.
\newblock Abelian covers of surfaces and the homology of the level {$L$}
  mapping class group.
\newblock {\em J. Topol. Anal.}, 3(3):265--306, 2011.

\bibitem[PW13]{PW}
A.~Putman and B.~Wieland.
\newblock Abelian quotients of subgroups of the mappings class group and higher
  {P}rym representations.
\newblock {\em J. Lond. Math. Soc. (2)}, 88(1):79--96, 2013.

\end{thebibliography}
	
\end{document}